\documentclass{amsart}

\usepackage{amsmath,amsthm,amsfonts,amssymb,mathrsfs}
\usepackage{dsfont}
\usepackage{hyperref}

\numberwithin{equation}{section}

\newtheorem{theorem}{Theorem}[section]
\newtheorem{corollary}[theorem]{Corollary}
\newtheorem{lemma}[theorem]{Lemma}
\newtheorem{proposition}[theorem]{Proposition}

\theoremstyle{definition}

\theoremstyle{remark}

\providecommand{\abs}[1]{\left\vert#1\right\vert}
\providecommand{\babs}[1]{\big\vert#1\big\vert}
\providecommand{\Babs}[1]{\Big\vert#1\Big\vert}
\providecommand{\nm}[1]{\left\Vert#1\right\Vert}

\providecommand{\bnm}[1]{\big\Vert#1\big\Vert}
\providecommand{\Bnm}[1]{\Big\Vert#1\Big\Vert}
\providecommand{\br}[1]{\left\langle #1 \right\rangle}

\def\ud{d}
\def\dt{\partial_t}
\def\p{\partial}
\def\ls{\lesssim}

\def\rt{\rightarrow}
\def\r{\mathbb{R}}
\def\no{\nonumber}
\def\ue{e}
\def\ui{\mathrm{i}}
\newcommand{\ep}{\varepsilon}

\def\R{\mathbb{R}}
\def\C{\mathbb{C}}
\def\Z{\mathbb{Z}}
\def\N{\mathbb{N}}

\def\ah{\alpha}
\def\bh{\beta}
\def\uh{u}
\def\vh{v}
\def\wah{\widehat{\ah}}

\def\pe{\phi}

\def\ph{\phi^h}
\def\wph{\widehat{\ph}}

\def\id{\mathds{1}}
\def\rr{\mathcal{R}}
\def\tt{\mathcal{T}}

\def\D{\partial_x^3}
\def\Lc{\Lambda_{{\rm c}}}
\def\Ld{\Lambda_{{\rm d}}}
\def\FDds{\abs{D_{{\rm d}}}}

\def\l{\ell}

\newcommand{\qtq}[1]{\quad\text{#1}\quad}
\newcommand{\eps}{\varepsilon}
\let\Re=\undefined
\DeclareMathOperator{\Re}{Re}
\let\Im=\undefined
\DeclareMathOperator{\Im}{Im}
\DeclareMathOperator{\supp}{supp}
\DeclareMathOperator{\sech}{sech}

\allowdisplaybreaks

\begin{document}

\title{The modified Korteweg--de Vries limit of the Ablowitz--Ladik system}

\author{Rowan Killip}
\address{Department of Mathematics, University of California, Los Angeles, CA 90095, USA}
\email{killip@math.ucla.edu}

\author{Zhimeng Ouyang}
\address{Department of Mathematics, University of Chicago, Chicago, IL 60637, USA}
\email{ouyangzm9386@uchicago.edu}

\author{Monica Visan}
\address{Department of Mathematics, University of California, Los Angeles, CA 90095, USA}
\email{visan@math.ucla.edu}

\author{Lei Wu}
\address{Department of Mathematics, Lehigh University}
\email{lew218@lehigh.edu}

\begin{abstract}
For slowly-varying initial data, solutions to the Ablowitz--Ladik system have been proven to converge to solutions of the cubic Schr\"odinger equation.  In this paper we show that in the continuum limit, solutions to the Ablowitz--Ladik system with $H^1$ initial data may also converge to solutions of the \emph{modified Korteweg--de Vries equation}.  To exhibit this new limiting behavior, it suffices that the initial data is supported near the inflection points of the dispersion relation associated with the Ablowitz--Ladik system.

Our arguments employ harmonic analysis tools, Strichartz estimates, and the conservation of mass and energy.  Correspondingly, they are applicable beyond the completely integrable models of greatest interest to us.
\end{abstract}

\maketitle

\setcounter{tocdepth}{1}
\makeatletter
\def\l@subsection{\@tocline{2}{0pt}{2.5pc}{5pc}{}}
\makeatother

\section{Introduction}
The Ablowitz--Ladik system
\begin{align}\label{AL}\tag{AL}
\ui\dt u_n &= -\big(u_{n-1}-2u_n+u_{n+1}\big)
+u_nv_n\big(u_{n-1}+u_{n+1}\big)
\end{align}
with $v_n = \pm \overline u_n$ describes the evolution of a field $u:\R\times \Z\to\C$.  For $v_n=\overline u_n$, the system is defocusing, while for $v_n=-\overline u_n$ it is focusing.

Ablowitz and Ladik introduced this model in \cite{MR0377223,MR0427867} as a discrete form of the one-dimensional cubic Schr\"odinger equation,
\begin{align}\label{NLS}\tag{NLS}
\ui\partial_t \psi = - \Delta \psi \pm 2 |\psi|^2\psi,
\end{align}
that preserves its complete integrability.

These authors already appreciated that solutions of \eqref{AL} with slowly-varying (smooth) initial data should constitute a good approximation to solutions to \eqref{NLS} in the continuum limit.  Subsequent investigations \cite{MR3939333,MR3009717} realized this idea rigorously for initial data in $H^1(\R)$.

In \cite{KOVW1} we proved that this continues to hold true for merely $L^2(\R)$ initial data.  Indeed, in \cite{KOVW1} we demonstrated that for $L^2$ initial data that combines slowly-varying and highly oscillatory components, the continuum limit of \eqref{AL} is a \emph{system} of two \eqref{NLS} equations.  To further explain this, it is convenient to pass to the Fourier representation $\widehat u(t,\theta)$ of the solution; see \eqref{ftc}.

In the system of \eqref{NLS} equations derived in \cite{KOVW1}, one component describes the dynamics of $\widehat u(t,\theta)$ at low frequencies $\theta\approx 0$, modulo $2\pi\Z$; the other describes the dynamics of the antipodal frequencies $\theta\approx \pi$.  Observe that these two regimes correspond to the critical points of the  dispersion relation $\omega_\text{AL}(\theta)=2-2\cos(\theta)$ associated to the linear part of \eqref{AL}.

As a long-term goal, we would like to understand the continuum limit of the \eqref{AL} dynamics for a class of initial data that energizes all frequencies equally --- a discrete form of white noise.  We are interested in this problem because we believe that it will provide insight for two famously thorny problems: (i) The construction of dynamics for \eqref{NLS} with white-noise distributed initial data; (ii) The construction of dynamics for the (one-dimensional) Landau--Lifschitz spin model in its Gibbs state.  For a further discussion of these connections, see \cite{MR4049393,Ishimori}. 

In this paper, we will be studying \eqref{AL} for initial data that excites frequencies near the inflection point $\theta=\pi/2$ of the dispersion relation.  For this purpose, it is convenient to introduce a change of unknown via
\begin{align}\label{1-2}
    \alpha_n(t) = (-\ui)^n\ue^{2\ui t}\uh_n(t) \qtq{and} \beta_n(t) = \pm\overline{\alpha}_n(t)=\ui^n\ue^{-2\ui t}\vh_n(t).
\end{align}
Rewritten in these variables, \eqref{AL} takes the form
\begin{align}\label{mAL} \tag{mAL}
    \dt\ah_n=-\big(1-\ah_n\bh_n\big)\big(\ah_{n+1}-\ah_{n-1}\big),
\end{align}
which we will refer to as the modified Ablowitz--Ladik system.

The factor $e^{2\ui t}$ in \eqref{1-2} removes the oscillation in time that originates from the non-vanishing of the dispersion relation $\omega_\text{AL} (\theta)=2-2\cos(\theta)$ at the inflection point $\theta=\pi/2$.  The factor $(-\ui)^n$ effects a Fourier rotation, bringing the inflection point to $\theta=0$.  By this reasoning, or by direct analysis, we find that the dispersion relation associated to \eqref{mAL} is
\begin{align}\label{mAL dr}
\omega_\text{mAL} (\theta)=2\sin(\theta).
\end{align}
We also see that the study of \eqref{AL} with initial data Fourier-concentrated around the inflection point corresponds to the study of \eqref{mAL} with slowly-varying initial data, that is, initial data Fourier-concentrated around $\theta=0$.  With this in mind, we will study the continuum limit of solutions to \eqref{mAL} with initial data of the following form: Given $\phi_0 \in H^1(\r)$, we let
\begin{align} \label{E:initial data}
    \alpha_n (0) = h \big[P_{\leq \frac\pi {2h}} \phi_0\big] (hn).
\end{align}
Here, $h$ is the length scale associated with the continuum approximation and $P_{\leq \frac\pi{2h}}$ is a smooth projection which eliminates frequencies $|\xi|> \frac{\pi}{h}$.  The role of this projection is to suppress aliasing and so simplify many subsequent formulae, beginning with \eqref{a0hat}.

The prefactor $h$ on RHS\eqref{E:initial data} ensures a balance between the dispersive and nonlinear parts of the evolution.  This is the same form of initial data considered~in~\cite{KOVW1}.  While the evolution equation is independent of $h$, $\alpha_n(t)$ depends strongly on $h$ through the choice of initial data.  However, given its role as the central object of our analysis, we choose not to clutter the notation by marking this dependence explicitly.

The dynamics \eqref{mAL} fits naturally into the Ablowitz--Ladik hierarchy and has Hamiltonian
\begin{align}\label{P defn}
    P(\alpha) = \Im \sum \beta_n\big(\ah_{n+1}-\ah_{n-1}\big),
\end{align}
with respect to the standard Poisson structure
\begin{align}\label{PB defn}
\bigl\{ F, G \bigr\} = -\ui \sum (1-\alpha_n\beta_n)\Bigl[ \tfrac{\partial F}{\partial \alpha_n}\tfrac{\partial G}{\partial \beta_n}
	- \tfrac{\partial F}{\partial \beta_n}\tfrac{\partial G}{\partial \alpha_n}\Bigr].
\end{align}
(Here partial derivatives are in the Wirtinger sense.)  We chose the notation $P(\alpha)$ for the Hamiltonian by analogy with conserved momentum for \eqref{NLS},
$$
P_\text{NLS}(\psi) = \Im \int \overline{\psi(x)}{\mkern 2mu}\psi'(x)\,dx,
$$
which is the generator of the translation symmetry of that equation.

The discrete-space models \eqref{AL} and \eqref{mAL} cannot have a continuous one-parameter group of translation symmetries like their continuum cousins.
From the Taylor expansion $\omega_\text{mAL} (\theta)=2\theta-\tfrac13\theta^3+O(\theta^5)$ of the dispersion relation, we see that beyond mere translation, \eqref{mAL} exhibits weak cubic dispersion.  Moreover, the leading order nonlinearity is cubic and of derivative type.  These features lead us naturally to regard
\begin{equation} \label{mKdV}\tag{mKdV}
    \dt\pe=-\p_x^3\pe\pm6\abs{\pe}^2\p_x\pe
\end{equation}
as describing the dynamics in the continuum limit, albeit in a moving reference frame to remove the leading translation.

To the best of our knowledge, this complex form of the modified Korteweg--de Vries equation was introduced by Hirota \cite{MR0338587}, who described its multisoliton solutions.  For real-valued initial data, \eqref{mKdV} reduces to the modified Korteweg--de Vries equation of Miura \cite{MR252825}.  The equation \eqref{mKdV} is known to be completely integrable and belongs to the same hierarchy as \eqref{NLS}.

For initial data as in \eqref{E:initial data}, it is clear that the characteristic timescale at which \eqref{mAL} exhibits non-trivial dynamics is $h^{-3}$.  Matching coefficients more carefully, we are lead to believe that the solution $\phi(t)$ of \eqref{mKdV} should describe the profile of $\alpha_n(3h^{-3}t)$.  To make this comparison precise, we must first convert the sequence $\alpha_n(3h^{-3}t)$ into a function on the line and then remove the overwhelming translation dynamics.

We pass from a sequence to a function on the real line using the operator $\mathcal R$ described in Lemma~\ref{L:R-opt}.  The choice embodies the Shannon Sampling Theorem: a function $f\in L^2(\R)$ with $\supp(\widehat f) \subseteq [-\tfrac\pi h, \tfrac\pi h]$ is uniquely determined by its values at the lattice points $h\Z$.

To remove the leading order translation dynamics associated to \eqref{mAL}, we must transport $\mathcal R \alpha_n(3h^{-3}t)$ to the left by a distance 
$2 h \cdot 3h^{-3}t$.  Note that the dispersion relation $\omega_\text{mAL} (\theta)$ has slope $2$ at $\theta=0$; this corresponds to a group velocity of $2$ in lattice variables or $2h$ in the continuum variable $x$.  The translation transformation will be denoted $\tt_t$:
\begin{align}\label{tt defn}
[\tt_t f](x) := f(x+6h^{-2}t).
\end{align}

Combining these operations, we find our candidate for an approximate solution to \eqref{mKdV} built from an exact solution $\alpha_n(t)$ of \eqref{mAL} to be
\begin{align}\label{psi-def}
\ph(t,x)= \bigl[(\tt_t\circ\rr) \alpha_n(3h^{-3}t)\bigr](x)
	= h^{-1}\int_{-\pi}^{\pi} \ue^{\ui h^{-1}\theta(x+6h^{-2}t)} \,\widehat{\ah}(3h^{-3}t,\theta)\,\tfrac{\ud\theta}{2\pi} ,
\end{align}
or equivalently,
\begin{align}\label{phi^h-hat}
    \widehat{\ph}(t,\xi)&=\ue^{6\ui h^{-2} t\xi }\,\widehat{\ah}(3h^{-3}t,h\xi)  \id_{[-\pi,\pi]}(h\xi) .
\end{align} 
Note that the passage from $\alpha_n(t)$ to $\ph(t)$ is reversible; indeed,
\begin{align}\label{aa 01}
    \ah_n(t)=h\ph\big(\tfrac{h^3}{3}t,hn-2ht\big).
\end{align}
Note also that the choice \eqref{E:initial data} of initial data corresponds to
$$
\ph(0,x) = [\mathcal R \alpha_n(0)](x)= P_{\leq \frac\pi {2h}} \phi_0(x).
$$

There is one more preliminary we must discuss:   Before we can compare solutions of \eqref{AL} and/or \eqref{mAL} with those of \eqref{mKdV}, we must show that such solutions actually exist.  In fact, since we claim convergence globally (but not uniformly) in time, we need to show the existence of \emph{global} solutions.

We review the $\ell^2(\Z)$ well-posedness of \eqref{mAL} in Proposition~\ref{Prop:l^2-bdd}.  Local well-posedness is elementary.  This is then made global-in-time using the conservation of mass \eqref{mass}.  Note that in the defocusing case, the mass is only defined for those $\alpha_n\in\ell^2(\Z)$ satisfying $\sup_n |\alpha_n| < 1$.  As we will show in Proposition~\ref{Prop:l^2-bdd}, this can be easily ensured by insisting that
\begin{equation}\label{h small}
0<h\leq h_0:=\min\Big\{1,\tfrac{1}{100\|\phi_0\|_{2}^2}\Big\}.
\end{equation}
As our goal is to send $h\to0$, we regard this as an efficient remedy.

On the continuum side, the arguments in \cite{MR1211741} show that \eqref{mKdV} is locally well-posed in $H^{1/4}(\R)$.  (The precise result formulated there is for the nonlinearity $\phi^2\partial_x\phi$.)  Global well-posedness in $H^1(\R)$ then follows from the conservation laws \eqref{311} associated with \eqref{mKdV}.  For further history and the sharp well-posedness result for this model in $H^s(\R)$, see \cite{HGKV}.  In Appendix~\ref{S:UCU}, we will show unconditional uniqueness for $L^\infty_t H^1_x$ solutions to \eqref{mKdV}; see Theorem~\ref{T:UCU} and the attendant discussion.

Our main result in this paper is the following: 

\begin{theorem}\label{thm:main}
Fix $\phi_0\in H^1(\R)$ and let $\phi\in C_t^{ } H^1_x(\R\times\R)$ denote the unique global solution of \eqref{mKdV} with this initial data.
For $h$ satisfying \eqref{h small}, let $\alpha_n(t)$ be the global solution to \eqref{mAL} with initial data specified by \eqref{E:initial data} and let $\phi^h:\R\times\R\to \C$ be the corresponding continuum representative of this solution built via \eqref{psi-def}. 
Then for any $T>0$, as $h\to 0$ we have the strong convergence
\begin{align}\label{258}
    \phi^h(t,x)\rt\phi(t,x) \quad\text{in \;$C_t^{ }H^s_x([-T,T]\times\r)$} \;\text{ for any \,$0\leq s<1$}.
\end{align}
Moreover if $t_n\to t$ and $h_n\to 0$ then 
\begin{align}\label{259}
    \phi^{h_n}(t_n,x) \rightharpoonup \phi(t,x) \quad\text{weakly in $H^1_x(\R)$}.
\end{align}
\end{theorem}

A major inspiration for literature on discrete approximations for Hamiltonian PDE has been developing and analyzing numerical methods \cite{MR3119633,MR4708776,MR1308108,MR2485456,MR2980459}.  Let us pause to view Theorem~\ref{thm:main} through this lens: It says that the discrete space approximation \eqref{mAL} provides a stable and convergent means of simulating both real- and complex-valued \eqref{mKdV} with low-regularity initial data.  The derivative nonlinearity makes this a significantly more difficult problem than for \eqref{NLS}, to which much of the literature just cited is devoted.

From our motivations described above, it was natural to formulate our main result as a description of the continuum dynamics.  However, this can easily be turned around to yield a result on the long-wave limit of the discrete model:  Combining \eqref{aa 01} and Lemma~\ref{L:sums to int} with the $s=0$ case of \eqref{258}, shows that for any $T>0$,
\begin{align}\label{258'}
    h^{\frac12}\bigl\| h^{-1} \alpha_n(3h^{-3}t) - \phi(t,nh-2ht) \bigr\|_{\ell^2_n} \rightarrow 0 \quad \text{uniformly for $|t|\leq T$}.
\end{align}

Much prior effort has been expended in understanding the long-wave limit for the FPUT chain, with the attitude that the complete integrability of the KdV provides an explanation for the recurrences observed there; see, for example, \cite{MR3366652,MR3327553,MR1870156}.

The problem with which we began our discussion was to describe solutions to \eqref{AL} with initial data whose Fourier support is centered around the inflection point of the dispersion relation.  Before discussing the proof of Theorem~\ref{thm:main}, let us recast the result in this light:

\begin{corollary}\label{C:AL}
Fix $\phi_0\in H^1(\R)$.  For $h$ satisfying \eqref{h small}, let $u_n(t)$ denote the global solution to \eqref{AL} with initial data
\begin{align} \label{E:initial data'}
    \uh_n (0) = \ui^n h \big[P_{\leq \frac\pi{2h}} \phi_0\big] (hn).
\end{align}
The solution $\phi(t)$ to \eqref{mKdV} with initial data $\phi_0$ describes the small-$h$ behavior of $u_n(t)$ in the following senses: The function
\begin{align}\label{psi-def'}
    \ph(t,x)=h^{-1}\ue^{6\ui h^{-3} t}\int_{-\pi}^{\pi} \ue^{\ui h^{-1}\theta(x+6h^{-2}t)} \,\widehat{u}(3h^{-3}t,\theta+\tfrac{\pi}{2})\,\tfrac{\ud\theta}{2\pi}
\end{align}
satisfies both \eqref{258} and \eqref{259}; moreover, as $h\to 0$,
\begin{align}
    h^{\frac12} \bigl\| h^{-1} u_n(3h^{-3}t) - \ui^n e^{-2it}\phi(t,nh-2ht) \bigr\|_{\ell^2_n} \rightarrow 0 \quad \text{uniformly for $|t|\leq T$}.
\end{align}
\end{corollary}

\subsection{Outline of the proof}
Overall, our strategy is based on a combination of compactness and uniqueness arguments: We show that the continuum manifestations $\phi^h(t)$ of the discrete solutions $\alpha_n(t)$ are bounded in $L^{\infty}_tH^1_x([-T,T]\times\R)$ and form a compact family in $C_t L^2_x([-T,T]\times\R)$.  These tasks are accomplished in Sections~\ref{Sec:Boundedness} and~\ref{S:4}, respectively.  Then in Proposition~\ref{Prop:convergence} we prove that any subsequential limit must be a solution of \eqref{mKdV} with initial data $\phi_0$ in the Duhamel sense, that is, \eqref{UCU Duhamel} holds.  By the bounds proved in Section~\ref{Sec:Boundedness}, this subsequential limit belongs to $L^\infty_t H^1_x$.  Now we may apply the following unconditional uniqueness result for solutions to \eqref{mKdV}:

\begin{theorem}\label{T:UCU}
Given $\phi_0\in H^1(\R)$ there is exactly one solution $\phi:\R\times\R\to\C$ to
\begin{align}\label{UCU Duhamel}
    \pe(t)=\ue^{-t\partial^3}\pe_0 \pm 6\! \int_0^t \!\ue^{-\ui(t-s)\partial^3} |\phi(s)|^2\phi'(s)\,\ud s
\end{align}
that belongs to $L^\infty_t H^1_x$. 
\end{theorem}

In this way, we see that all subsequential limits agree and consequently, $\phi^h$ converges in $C_t^{ } L^2_x([-T,T]\times\R)$ to this unique solution to \eqref{mKdV}.  It is then an easy matter to upgrade the modes of convergence to those presented in Theorem~\ref{thm:main} using the uniform bounds of Section~\ref{Sec:Boundedness}. 

This leaves us to prove Theorem~\ref{T:UCU}.  As noted earlier, the local-in-time existence of solutions with initial data $\phi_0\in H^s(\R)$ was proved already in \cite{MR1211741} for $s\geq\frac14$.  In  \cite{MR1211741}, solutions were constructed via contraction mapping in a Banach space that imposes additional spacetime bounds.  Global well-posedness in $H^1(\R)$ then follows through the use of the conservation laws:
\begin{align}\label{311}
\int |\phi(x)|^2\,dx \qtq{and} \int |\phi'(x)|^2 \pm |\phi(x)|^4\,dx.
\end{align}
The central matter to be discussed is uniqueness for $L^\infty_t H^1_x$ solutions (without additional spacetime bounds).  This is known as \emph{unconditional uniqueness}.  For real-valued solutions to \eqref{mKdV}, the paper \cite{Kwon.Oh.Yoon2020} demonstrated such unconditional uniqueness in $C_t^{ } H_x^s$ with $s>\frac{1}{4}$.  We are not aware of any such results for complex-valued \eqref{mKdV} and so provide a proof of Theorem~\ref{T:UCU} in Appendix~\ref{S:UCU}.  Concretely, we show that $L^\infty_t H^1_x$ solutions necessarily satisfy certain spacetime bounds that allow us to prove uniqueness via a Gronwall argument. The methods we employ here are elementary and can be applied to a large class of gKdV-like equations, such as combined nonlinearities of at least cubic type; moreover, the derivative can land on any factor in each nonlinearity.

\subsection*{Acknowledgements}

Rowan Killip was supported by NSF grant DMS-2154022. Zhimeng Ouyang was supported by the NSF Postdoctoral Fellowship DMS-2202824. Monica Visan was supported by NSF grant DMS-2054194. Lei Wu was supported by NSF grant DMS-2104775.

\section{Preliminaries}

Throughout this paper, $C$ will denote a constant that does not depend on the initial data or on $h$, and which may vary from one line to another.
We write $A \lesssim B$ or $B\gtrsim A$ whenever $A\leq CB$ for such a constant $C>0$.  We write $A\simeq B$ whenever $A\lesssim B$ and $B\lesssim A$.  
If $C$ depends on some additional parameters, we will indicate this with subscripts.

Throughout this paper, we will employ the Lebesgue spaces $L^p$ and $\l^p$, as well as the spaces $L_t^q L_x^p$ and $L^q_t\l^p_n$ equipped with the norms
\begin{align*}
    \|F\|_{L_t^q L_x^p(\R\times\R)} &:=\bigg[\int_{\R}\Bigl(\int_{\R} \big|F(t,x)\big|^p \,dx \Bigr)^{\frac{q}{p}} \,dt\bigg]^{\frac{1}{q}},\\
    \nm{\ah}_{L^q_t\l^p_n(\R\times\Z)} &:= \bigg[\int_{\R}\Big(\sum_{n\in\Z}\babs{\ah_n(t)}^p\Big)^{\frac{q}{p}}d t\bigg]^{\frac{1}{q}},
\end{align*}
with the usual modifications when $q$ or $p$ is $\infty$, or when the domain is replaced by some smaller subset.
When $p=q$ we abbreviate $L_t^q L_x^p$ as $L^q_{t,x}$.

Our conventions for the Fourier transform on $\r$ are as follows:
\begin{align}\label{ftr}
    \widehat{f}(\xi) = \int_\R f(x) e^{-i x\xi}\,dx \quad\text{so that}\quad f(x) =\int_\R\widehat{f}(\xi) e^{i x\xi}\,\tfrac{d\xi}{2\pi},
\end{align}
while in the discrete case, we employ
\begin{align}\label{ftc}
\widehat a(\theta) = \sum_{n\in \Z} a_n e^{-i n\theta} \quad\text{so that}\quad a_n = \int_{-\pi}^{\pi} \widehat a(\theta) e^{i n\theta}\, \tfrac{d\theta}{2\pi}.
\end{align}
With these definitions, the Plancherel/Parseval identities read
\begin{align}\label{ftp}
    \int_\R \overline{g(x)} {\mkern 2mu} f(x)\, dx = \int_\R \overline{\widehat{g}(\xi)} {\mkern 2mu} \widehat{f}(\xi)\,\tfrac{d\xi}{2\pi}
    	\quad\text{and}\quad
    		\sum_{n\in \Z} \overline{b_n} {\mkern 2mu} a_n = \int_{-\pi}^{\pi} \overline{\widehat{b}(\theta)} {\mkern 2mu} \widehat{a}(\theta) \,\tfrac{d\theta}{2\pi}.
\end{align}

For $s\in \R$, we write  $\abs{\nabla}^s$ and $\br{\nabla}^s$ for the Fourier multiplier operators with symbols $\abs{\xi}^s$ and $\br{\xi}^s$, respectively. 
The inhomogeneous and homogeneous Sobolev spaces $H^s(\R)$ and $\dot H^s(\R)$ are defined as the closure of Schwartz functions under the norms
\begin{align*}
\|f\|_{H^s(\R)}^2 &:= \big\|\br{\nabla}^s\!f\big\|_{L^2(\R)}^2 = \int_{\R} \big(1+|\xi|^2\big)^s \babs{\widehat{f}(\xi)}^2 \,\tfrac{d\xi}{2\pi},\\
\|f\|_{\dot{H}^s(\R)}^2 &:= \big\|\abs{\nabla}^s\!f\big\|_{L^2(\R)}^2 = \int_{\R} |\xi|^{2s} \babs{\widehat{f}(\xi)}^2 \,\tfrac{d\xi}{2\pi} .
\end{align*}
Clearly, for $s=1$ we have $\|f\|_{H^1}^2 = \|f\|_{L^2}^2 + \|f\|_{\dot{H}^1}^2$.

Let $\varphi$ be a smooth function supported in the ball $|\xi| \leq 2$ such that $\varphi(\xi)=1$ for all $|\xi| \leq 1$. For each dyadic number $N \in 2^\Z$ we define the Littlewood--Paley operators
\begin{align*}
\widehat{P_{\leq N}f}(\xi) \!:=  \varphi\bigl(\tfrac{\xi}N\bigr)\widehat f (\xi), \qquad \widehat{P_N f}(\xi) \!:=  \bigl[\varphi\bigl(\tfrac{\xi}N\bigr) - \varphi \bigl(\tfrac{2 \xi}N\bigr)\bigr]\widehat f (\xi), \qquad P_{> N} :=  1-P_{\leq N}.
\end{align*}
We will frequently write $f_{\leq N}$ for $P_{\leq N} f$ and similarly for the other operators.

The Littlewood--Paley operators commute with derivative operators.  They are bounded on $L^p(\R)$ and $H^s(\R)$ for every $1 \leq p \leq\infty$ and $s\in \R$.  They also satisfy the following Sobolev and Bernstein estimates:
\begin{align*}
\| |\nabla|^{ s} P_N f\|_{L^p_x} \simeq N^{s} \| P_N f \|_{L^p_x}, \qquad \|P_N f\|_{L^q_x} \lesssim N^{\frac{1}{p}-\frac{1}{q}} \| P_N f\|_{L^p_x},
\end{align*}
whenever $s\in \R$ and $1 \leq p \leq q \leq \infty$.

Let us now discuss the passage between functions on the line and those on the $h\Z$ lattice.  We must examine both directions: the initial data $\alpha_n(0)$ is constructed from the initial data for \eqref{mKdV} via \eqref{E:initial data}; the resulting solutions $\alpha_n(t)$ of \eqref{mAL} are then transported back to the line via \eqref{psi-def}.  We begin with the mathematical embodiment of the Shannon Sampling Theorem:

\begin{lemma}\label{L:sums to int}
If $f,g \in L^2(\R)$ satisfy $\supp(\widehat f\,)\subseteq\bigl[-\frac\pi h,\frac \pi h\bigr]$ and $\supp(\widehat g)\subseteq\bigl[-\frac\pi h,\frac \pi h\bigr]$, then
\begin{align}\label{327}
\int_{\R} \overline{g(x)} f(x) \, dx=\int \overline{\widehat g(\xi)} \widehat f(\xi) \tfrac{d\xi}{2\pi} =  h\sum_{n\in\Z} \overline{g(nh)}\,f(nh).
\end{align}
In particular, for such functions $f$ and any $x\in\R$,
\begin{align}\label{328}
f(x) = \int e^{\ui x\xi} \widehat f(\xi) \tfrac{d\xi}{2\pi} = \sum_n \tfrac{\sin(\pi[x-nh]/h)}{\pi[x-nh]/h} f(nh).
\end{align}
\end{lemma}

\begin{proof}
Both identities in \eqref{327} are consequences of the isometry property \eqref{ftp} of the Fourier transforms defined in \eqref{ftr} and \eqref{ftc}.
To deduce \eqref{328}, we apply \eqref{327} with $\widehat g(\xi) = e^{-\ui\xi x} \id_{[-\pi,\pi]}(h\xi)$.
\end{proof}

As an immediate corollary of this lemma, we find the basic mapping properties of the operator we are using to pass from sequences to functions on the real line:

\begin{lemma}\label{L:R-opt}
The operator $\mathcal R:\ell_n^2(\Z)\to L^2_x(\R)$ defined by
\begin{align}\label{R def}
[\mathcal R c] (x) := \tfrac1h \! \int_{-\pi}^{\pi} e^{\ui x\theta/h}\, \widehat c(\theta)\tfrac{d\theta}{2\pi}
	= \int_{-\pi/h}^{\pi/h} e^{\ui x\xi} \, \widehat c(h\xi)\tfrac{d\xi}{2\pi}
	= \sum_n \tfrac{\sin(\pi[x-nh]/h)}{\pi[x-nh]} c_n
\end{align}
is bounded.  Indeed, it satisfies 
\begin{align}
    \|\mathcal Rc \|_{L^2_x(\R)}^2 =  h^{-1} \|c \|_{\ell_n^2(\Z)}^2.
\end{align}
\end{lemma}

Note that \eqref{psi-def} can be rewritten as
\begin{align} \label{psi-def-R}
\ph(t,x)=[\rr\ah](3h^{-3}t,x+6h^{-2}t)
	= (\tt_t \circ\rr)\bigl[\ah(3h^{-3}t)\bigr],
\end{align}
where $\tt_t$ denotes the operator of translation by $6h^{-2}t$.

Observe also that
\begin{equation}\label{357}
\supp\bigl( \widehat{\mathcal R c}\bigr) \subseteq  \bigl[-\tfrac\pi h,\tfrac \pi h\bigr]
	\qtq{and that} [\mathcal R c] (nh) = h^{-1}c_n.
\end{equation}
In view of the reconstruction formula \eqref{328}, these properties uniquely characterize the function $[\mathcal Rc](x)$.

\begin{lemma}\label{L:abc}
Suppose  $a_n, b_n, c_n \in \ell^2(\Z)$ all have Fourier support contained in the arc $[-1,1]$.  Then
\begin{equation}\label{E:abc}
\rr\bigl[ a b c\bigr](x) = h^2 \, \rr[ a ](x) \cdot \rr[ b ](x) \cdot \rr[ c ](x).
\end{equation}
\end{lemma}

\begin{proof}
From the Fourier support restrictions on the three sequences, \eqref{R def} shows that each of the three functions on RHS\eqref{E:abc} has Fourier support in the interval $[-h^{-1},h^{-1}]$.  Thus the Fourier transform of their product is supported in $[-\pi h^{-1},\pi h^{-1}]$.  This interval also supports the Fourier transform of $\rr[abc]$.  In view of \eqref{328}, the equality \eqref{E:abc} for general $x$ follows from the case $x\in h\Z$, which in turn follows immediately from \eqref{357}.
\end{proof}

\begin{lemma}[Preservation of locality]\label{lem:reconstruction2}
Let $c \in \ell^2_n$ be a sequence supported in the interval $[\tau-\frac Lh,\tau+\frac Lh]$ for some $\tau\in\R$ and $L>0$.
Then for any $L'>0$ we have
\begin{align}\label{E:tighty}
\int_{|x-\tau h|\geq L+L'} \bigl|[\mathcal R c](x)\bigr|^2\,dx \lesssim \frac{L}{h L'} \| c \|_{\ell^2_n}^2 .
\end{align}
\end{lemma}

\begin{proof}
For $\abs{n-\tau}\leq h^{-1}L$ and $|x-\tau h|\geq L+L'$ we have $|x-nh|\geq L'$. Applying the Minkowski and H\"older inequalities, we deduce
\begin{align*}
    \int_{|x-\tau h|\geq L+L'} \bigl|[\mathcal R c](x)\bigr|^2\,dx
    &\ls \left\{\!\sum_{\abs{n-\tau}\leq h^{-1}L} \!\!\!\!\abs{c_n} \bigg(\int_{|x-nh|\geq L'} \Bigl| \frac{\sin(\tfrac{\pi}{h}(x-nh))}{\pi(x-nh)}\Bigr|^2\,dx\bigg)^{\frac{1}{2}}\right\}^{2} \\
    &\ls \frac{1}{L'} \bigg(\sum_{\abs{n-\tau}\leq h^{-1}L} \abs{c_n} \bigg)^{2}\\
    &\ls \frac{1}{L'} \Big( h^{-\frac12} L^\frac12 \| c \|_{\ell^2_n}\Big)^{2}
    = \frac{L}{h L'} \| c \|_{\ell^2_n}^2. \qedhere
\end{align*}
\end{proof}

Let us now turn our attention to the mapping of initial data, as given in \eqref{E:initial data}. The $L^2\to\ell^2$ boundedness of this mapping is evident from Lemma~\ref{L:sums to int}:
\begin{align}\label{alpha0 mass}
\|\alpha(0)\|_{\ell_n^2}^2\leq h\|\phi_0\|_{L^2}^2.
\end{align}

Applying the Poisson summation formula 
\begin{align*}
    \sum_{m\in\Z}\widehat{f}\left(\tfrac{\theta+2\pi m}{h}\right) = \sum_{n\in\Z} hf(nh)e^{-i n\theta}
\end{align*}
to the function $f = P_{\leq \frac{\pi}{2h}}\phi_0$, we find
\begin{align} \label{a0hat}
\widehat{P_{\leq \frac{\pi}{2h}}\phi_0}\bigl(\tfrac\theta h\bigr) = \id_{[-\pi,\pi]}(\theta) \sum \alpha_n(0) e^{-\ui n\theta} =  \id_{[-\pi,\pi]} \widehat\alpha(0,\theta).
\end{align}

For our arguments in Section~\ref{S:4}, we need to control the spatial distribution of $\alpha_n(0)$ uniformly as $h\to 0$.  This is the topic of our next lemma.

\begin{lemma}[Locality of $\alpha_n(0)$]\label{L:locA}
Fix $\phi_0\in L^2$ and let $\alpha_n (0) = h \big[P_{\leq \frac\pi {2h}} \phi_0\big] (hn)$, as in \eqref{E:initial data}.  Then
\begin{align}\label{E:locA}
\bigl\|\alpha_n(0)\bigr\|^2_{\ell^2(|n|\geq m)} \lesssim h \bigl\| \mathcal M \phi_0\bigr\|_{L^2(|x|\geq [m-\frac12] h)}^2 ,
\end{align}
uniformly for $h\in(0,1]$ and $m\in \N$.  Here, $\mathcal M\phi_0$ denotes the Hardy--Littlewood maximal function of $\phi_0$.
\end{lemma}

\begin{proof}
For any Schwartz function $\check\varphi$, we have
$$
N \bigl| \check\varphi(N[y+x]) \bigr| \lesssim \frac{N}{1+N^2y^2} \qtq{uniformly for} |x|\leq N^{-1} \qtq{and} 0<N<\infty.  
$$
In particular, we may apply this to the convolution kernel associated to Littlewood--Paley projections and we may choose $N=\frac\pi {2h}$.

The specific instance to which we apply this observation is
\begin{align*}
h \big[P_{\leq \frac\pi {2h}} \phi_0\big] (hn) &= h \int N\check\varphi(Ny) \phi_0(nh-y)\,dy \\
&= \int_{-\frac h2}^{\frac h2}\int N\check\varphi(N[y+x]) \phi_0(nh-y-x)\,dy\,dx.
\end{align*}
In this way, we obtain the estimate
\begin{align*}
\bigl|\alpha_n(0)\bigr| \lesssim \int_{-\frac h2}^{\frac h2}\int \frac{N|\phi_0(nh-y-x)|}{1+N^2y^2} \,dy\,dx
	\lesssim \int_{-\frac h2}^{\frac h2} [\mathcal M \phi_0](nh-x) \,dx .
\end{align*}
Applying the Cauchy--Schwarz inequality, we deduce that
\begin{align*}
\bigl|\alpha_n(0)\bigr|^2 \lesssim h \int_{-\frac h2}^{\frac h2} \bigl| [\mathcal M \phi_0](nh-x) \bigr|^2 \,dx 
	= h \int_{nh-\frac h2}^{nh+\frac h2} \bigl| [\mathcal M \phi_0](x) \bigr|^2 \,dx.
\end{align*}
The estimate \eqref{E:locA} now follows by summing over $|n|\geq m$.
\end{proof}

Next we record the Strichartz estimates for the Airy propagator and their discrete analogues.

\begin{lemma}[Strichartz estimates; \cite{GV,Kenig.Ponce.Vega1991}] \label{P:Strichartz}
If $I$ is a time interval containing zero, then 
\begin{align} \label{Strichartz-ineq}
\Big\|\int_0^t e^{ -(t-s)\D}F(s)\,\ud s \Big\|_{L_t^\infty L^2_x(I\times\R)} 
\ls \bnm{|\nabla|^{-\frac{1}{8}}F}_{L_t^{\frac{8}{7}}L_x^{\frac{4}{3}}(I\times\R)}
\end{align}
and 
\begin{align} \label{Strichartz-ineq-d}
\Big\|\int_0^t e^{ \ui(t-s)\Ld}F(s)\,\ud s \Big\|_{L_t^\infty \l_n^2(I\times\Z)} 
\ls \bnm{\FDds^{-\frac{1}{8}}F}_{L_t^{\frac{8}{7}}\l_n^{\frac{4}{3}}(I\times\Z)},
\end{align}
where $\Ld$ and $\FDds^{-\frac18}$ are discrete operators with Fourier symbols given by $\Ld(\theta)= -2\sin(\theta)$ and $\FDds(\theta)=|\sin(\theta)|$.
\end{lemma}

\begin{proof}
The estimate \eqref{Strichartz-ineq} is a particular example of the Strichartz estimates for the Airy propagator $e^{-t\D}$ derived in \cite{GV,Kenig.Ponce.Vega1991}. The same arguments show that the discrete propagator $e^{\ui t\Ld}$ enjoys the parallel Strichartz estimate \eqref{Strichartz-ineq-d}.  Note that in the discrete setting, the dispersion relation $\Ld(\theta)$ has inflection points at both $\theta=0$ and $\theta=\pm\pi$ with $\babs{\Ld''(\theta)}\simeq\abs{\sin(\theta)}$, which explains the presence of $\FDds^{-\frac18}$ on the right-hand side of \eqref{Strichartz-ineq-d}.
\end{proof}

\begin{lemma}[Paraproduct estimates] \label{L:paraproduct}
For functions on the real line we have
\begin{align} \label{paraproduct-1}
    \bnm{|\nabla|^{-\frac{1}{8}}(fgh)}_{L^{\frac{4}{3}}(\R)} \ls 
    \bnm{|\nabla|^{-\frac{1}{8}}f}_{L^{2}(\R)}\bnm{|\nabla|^{\frac{3}{8}}g}_{L^{2}(\R)}\bnm{|\nabla|^{\frac{3}{8}}h}_{L^{2}(\R)},
\end{align}
while in the discrete setting we have the parallel estimate
\begin{align} \label{paraproduct-2}
    \bnm{\FDds^{-\frac{1}{8}}(fgh)}_{\ell^{\frac{4}{3}}(\Z)} \ls 
    \bnm{\FDds^{-\frac{1}{8}}f}_{\ell^{2}(\Z)}\bnm{\FDds^{\frac{3}{8}}g}_{\ell^2(\Z)}\bnm{\FDds^{\frac{3}{8}}h}_{\ell^2(\Z)},
\end{align}
where $\FDds$ is the discrete operator with Fourier symbol $\FDds(\theta)=|\sin(\theta)|$.
\end{lemma}

\begin{proof}
We begin with \eqref{paraproduct-1}, which we prove arguing by duality.  Observe that this estimate is equivalent to
\begin{align}\label{para1}
\bigl\langle|\nabla|^{\frac{1}{8}}f \cdot |\nabla|^{-\frac{3}{8}}g \cdot |\nabla|^{-\frac{3}{8}} h,  |\nabla|^{-\frac{1}{8}} \phi \bigr\rangle \lesssim \|f\|_{L^2}\|g\|_{L^2}\|h\|_{L^2}\|\phi\|_{L^4}.
\end{align}
We begin by decomposing into Littlewood--Paley pieces so that 
\begin{align}
\text{LHS\eqref{para1}}&\lesssim \sum_{N_1\sim N_2} \int \big||\nabla|^{-\frac{1}{8}} \phi_{N_1}\bigr| \big||\nabla|^{\frac{1}{8}} f_{N_2}\bigr|
\big|P_{\leq \tfrac{N_1\vee N_2}{16} }\bigl( |\nabla|^{-\frac{3}{8}}g \cdot |\nabla|^{-\frac{3}{8}} h \bigr)\bigr|\notag\\
&\quad+ \sum_{N_3\sim N_1\vee N_2}  \int \big||\nabla|^{-\frac{1}{8}} \phi_{N_1}\bigr| \big||\nabla|^{\frac{1}{8}} f_{N_2}\bigr| 
\big|P_{N_3}\bigl( |\nabla|^{-\frac{3}{8}}g \cdot |\nabla|^{-\frac{3}{8}} h \bigr)\bigr|,
\label{10:04}
\end{align}
where we use the notation $N_1 \vee N_2 :=\max\{N_1, N_2\}$.

Recall that smooth Littlewood--Paley projections to low frequencies are bounded by the Hardy--Littlewood maximal function, that is,
\begin{align}\label{L to M}
\bigl|P_{\leq N} (\psi) \bigr|\lesssim \mathcal{M} (\psi) \quad\text{uniformly for $N\in 2^\Z$}.
\end{align}
Combining this with Cauchy--Schwarz in $N_1,N_2$,  H\"older, the $L^4$-boundedness of the Hardy--Littlewood maximal function, Littlewood--Paley square function estimates, and Sobolev embedding, we may bound
\begin{align*}
&\sum_{N_1\sim N_2} \int \big||\nabla|^{-\frac{1}{8}} \phi_{N_1}\bigr| \big||\nabla|^{\frac{1}{8}} f_{N_2}\bigr|
	\big|P_{\leq \tfrac{N_1\vee N_2}{16} }\bigl( |\nabla|^{-\frac{3}{8}}g \cdot |\nabla|^{-\frac{3}{8}} h \bigr)\bigr|\\
&\lesssim \int \Bigl\{ \sum_{N_1} \bigl| N_1^{\frac18} |\nabla|^{-\frac18} \phi_{N_1}\bigr|^2\Bigr\}^{\frac12} \Bigl\{ \sum_{N_2} \bigl| N_2^{\frac18} |\nabla|^{-\frac18} f_{N_2}\bigr|	^2\Bigr\}^{\frac12}\mathcal{M} \bigl( |\nabla|^{-\frac{3}{8}}g \cdot |\nabla|^{-\frac{3}{8}} h \bigr) \\
&\lesssim \Bigl \|\Bigl\{ \sum_{N_1} \bigl| N_1^{\frac18} |\nabla|^{-\frac18} \phi_{N_1}\bigr|^2\Bigr\}^{\frac12} \Bigr\|_{L^4} \Bigl\|\Bigl\{ \sum_{N_2} \bigl| N_2^{\frac18} |\nabla|^{-\frac18} f_{N_2}\bigr|^2\Bigr\}^{\frac12}  \Big\|_{L^2} \Bigr\| \mathcal{M} \bigl( |\nabla|^{-\frac{3}{8}}g \cdot |\nabla|^{-\frac{3}{8}} h \bigr) \Bigr\|_{L^4}\\
&\lesssim \|\phi\|_{L^4}\|f\|_{L^2}\bigl\| |\nabla|^{-\frac{3}{8}}g \cdot |\nabla|^{-\frac{3}{8}} h \bigr\|_{L^4}\\
&\lesssim \|\phi\|_{L^4}\|f\|_{L^2}\bigl\| |\nabla|^{-\frac{3}{8}}g\bigr\|_{L^8}\bigr\| |\nabla|^{-\frac{3}{8}} h \bigr\|_{L^8}\\
&\lesssim  \|\phi\|_{L^4}\|f\|_{L^2}\|g\|_{L^2}\|h\|_{L^2},
\end{align*}
which is acceptable.

For the second sum on the right-hand side of \eqref{10:04}, we distinguish two cases:

\noindent\textbf{Case 1:} $N_1\geq N_2$, which implies $N_1\sim N_3$.  In this case, we further decompose the functions $g, h$ in frequency using Littlewood--Paley projections $P_{M_1}, P_{M_2}$.  By symmetry, we may assume that $M_1\geq M_2$ and so $M_1\gtrsim N_3$.  Using H\"older and Bernstein inequalities, followed by Schur's test, we may bound
\begin{align*}
&\sum_{N_3\sim N_1\geq N_2}  \int \big||\nabla|^{-\frac{1}{8}} \phi_{N_1}\bigr| \big||\nabla|^{\frac{1}{8}} f_{N_2}\bigr|
\big|P_{N_3}\bigl( |\nabla|^{-\frac{3}{8}}g \cdot |\nabla|^{-\frac{3}{8}} h \bigr)\bigr|\\
&\qquad\lesssim \sum_{N_2\leq N_1\lesssim M_1\geq M_2} \bigl\| |\nabla|^{-\frac{1}{8}} \phi_{N_1}\bigr\|_{L^4} \bigl\| |\nabla|^{\frac{1}{8}} f_{N_2}\bigr\|_{L^4} \bigl\| |\nabla|^{-\frac{3}{8}}g_{M_1} \bigr\|_{L^2} \bigl\| |\nabla|^{-\frac{3}{8}}h_{M_2} \bigr\|_{L^\infty}\\
&\qquad\lesssim \sum_{N_2\leq N_1\lesssim M_1\geq M_2} N_1^{-\frac18} \|\phi_{N_1}\|_{L^4} N_2^{\frac38}\|f_{N_2}\|_{L^2} M_1^{-\frac38}\|g_{M_1}\|_{L^2} M_2^{\frac18}\|h_{M_2}\|_{L^2}\\
&\qquad\lesssim \|\phi\|_{L^4}\|h\|_{L^2} \sum_{N_2\lesssim M_1} \bigl(\tfrac{N_2}{M_1}\bigr)^{\frac14}\|f_{N_2}\|_{L^2}\|g_{M_1}\|_{L^2}\\
&\qquad\lesssim \|\phi\|_{L^4}\|f\|_{L^2}\|g\|_{L^2}\|h\|_{L^2},
\end{align*}
which is acceptable.

\noindent\textbf{Case 2:} $N_1\leq N_2$, which implies $N_2\sim N_3$.  Arguing as in Case 1, we may bound
\begin{align*}
&\sum_{N_3\sim N_2\geq N_1}  \int \big||\nabla|^{-\frac{1}{8}} \phi_{N_1}\bigr| \big||\nabla|^{\frac{1}{8}} f_{N_2}\bigr|
\big|P_{N_3}\bigl( |\nabla|^{-\frac{3}{8}}g \cdot |\nabla|^{-\frac{3}{8}} h \bigr)\bigr|\\
&\qquad\lesssim \sum_{N_1\leq N_2\lesssim M_1\geq M_2} \bigl\| |\nabla|^{-\frac{1}{8}} \phi_{N_1}\bigr\|_{L^\infty} \bigl\| |\nabla|^{\frac{1}{8}} f_{N_2}\bigr\|_{L^2} \bigl\| |\nabla|^{-\frac{3}{8}}g_{M_1} \bigr\|_{L^2} \bigl\| |\nabla|^{-\frac{3}{8}}h_{M_2} \bigr\|_{L^\infty}\\
&\qquad\lesssim \sum_{N_1\leq N_2\lesssim M_1\geq M_2} N_1^{\frac18} \|\phi_{N_1}\|_{L^4} N_2^{\frac18}\|f_{N_2}\|_{L^2} M_1^{-\frac38}\|g_{M_1}\|_{L^2} M_2^{\frac18}\|h_{M_2}\|_{L^2}\\
&\qquad\lesssim \|\phi\|_{L^4}\|h\|_{L^2} \sum_{N_2\lesssim M_1} \bigl(\tfrac{N_2}{M_1}\bigr)^{\frac14}\|f_{N_2}\|_{L^2} \|g_{M_1}\|_{L^2}\\
&\qquad\lesssim \|\phi\|_{L^4}\|f\|_{L^2}\|g\|_{L^2}\|h\|_{L^2},
\end{align*}
which is also acceptable.  This completes the proof of \eqref{para1}, from which \eqref{paraproduct-1} follows by duality.

The proof of \eqref{paraproduct-2} follows a parallel path.  The key ingredient here is the development of a suitable Littlewood--Paley theory that allows for the implementation of the argument used above in the discrete setting. To this end, we choose a smooth even function $\varphi\in C^\infty(\R)$ supported in $(-\frac{3\pi}4, \frac{3\pi}4)$ such that
$$
\sum_{n\in \Z} \varphi(\theta +\pi n)=1 \qquad \text{for all} \qquad \theta\in \R.
$$
We then define
$$
\varphi^+(\theta) := \sum_{n\in \Z} \varphi (\theta+2\pi n) \qquad\text{and} \qquad \varphi^-(\theta) :=\varphi^+(\theta-\pi).
$$ 

For a dyadic number $N\leq 1$, we define the Littlewood--Paley projections to low frequencies localized near $\theta = 0 $  (mod $2\pi$) and $\theta= \pi$ (mod $2\pi$) via
$$
\widehat {P_{\leq N}^{+} f} (\theta):= \varphi^{+}(\tfrac{\theta}N\bigr)\widehat f(\theta) \qquad \text{and} \qquad \widehat {P_{\leq N}^{-} f} (\theta):= \varphi^{-}(\tfrac{\theta}N\bigr)  \widehat f(\theta),
$$
respectively.  Defining $P_N^{\pm} := P_{\leq N}^{\pm} - P_{\leq \frac N2}^{\pm}$, we have the decomposition
\begin{align}\label{LP decomp}
f = \sum_{\sigma=\pm } \sum_{N\leq 1} P_N^\sigma f.
\end{align}

The very first step in the proof of \eqref{paraproduct-1} was to decompose all functions into their Littlewood--Paley pieces and observe that the only summands that contribute to the left-hand side of \eqref{para1} are those where the two highest frequencies are comparable.   This fact carries over when using the Littlewood--Paley decomposition \eqref{LP decomp} for any particular assignment of sign parameters $\sigma$ to the four functions.  Note that there are only finitely many choices of such sign parameters.

It remains to review discrete analogues of the basic functional estimates used previously.   Because of the uniform smoothness built into these Littlewood--Paley projections, Bernstein and Mikhlin multiplier estimates follow from the standard arguments. The Mikhlin multiplier estimates then guarantee boundedness of the Littlewood--Paley square function corresponding to the decomposition \eqref{LP decomp}.  Lastly, the Sobolev embedding inequality
$$
\| \FDds^{-\frac38} f\|_{\ell^8}\lesssim \|f\|_{\ell^2}
$$
follows from the Hardy--Littlewood--Sobolev inequality and the observation
\begin{equation*}
\Bigl|\int_{-\pi}^\pi |\sin(\theta)|^{-\frac38} e^{im\theta}\,\tfrac{d\theta}{2\pi}\Bigr|\lesssim \langle m\rangle^{-\frac58} \qquad\text{uniformly for $m\in \Z$}.
\qedhere
\end{equation*}
\end{proof}

\section{Uniform Boundedness in \texorpdfstring{$L_t^\infty H_x^1$}{}} \label{Sec:Boundedness}

As a completely integrable system, \eqref{mAL} enjoys infinitely many conservation laws. However, we will only employ the following two:
\begin{align}
\textbf{mass} \qquad   &M(\ah):=-\sum_{n\in\Z}\ln\big(1-\ah_n\bh_n\big),\label{mass}\\
\textbf{energy} \qquad     &E(\ah):=-\sum_{n\in\Z}\Big\{\ah_n\bh_{n+1}+\ah_{n+1}\bh_n+2\ln\big(1-\ah_n\bh_n\big)\Big\}\label{energy}.
\end{align}

\begin{proposition}[GWP for \eqref{mAL}] \label{Prop:l^2-bdd}
Given $0<h\leq h_0:=\min\big\{1,\frac{1}{100\|\phi_0\|_{2}^2}\big\}$, the equation \eqref{mAL} with initial data \eqref{E:initial data} admits a unique global solution in $C_t\ell_n^2$. Moreover,
\begin{equation}\label{E:mass bound}
\bigl| M(\alpha(t))\bigr|  \simeq \nm{\ah(t)}_{\l_n^2}^2\lesssim h  \|\phi_0\|_{L^2}^2  
\qquad\text{uniformly for $t\in \R$}.
\end{equation}
\end{proposition}

\begin{proof}
The claims follow from the arguments we used  in \cite[Proposition 3.1]{KOVW1} to derive the analogous result for the \eqref{AL} system.  For completeness, we present the details.

Local well-posedness in $\ell^2$ is guaranteed by Picard's theorem.  To show that these local solutions can be extended globally in time, we will prove an a priori $\ell^2$-bound for solutions.

From the power series expansion of the logarithm, we have
\begin{align}\label{8:52}
\Bigl| \bigl| M(\alpha(t))\bigr|  - \nm{\ah(t)}_{\l_n^2}^2 \Bigr| \leq \sum_{\ell=2}^{\infty}\tfrac{1}{\ell}\nm{\ah(t)}_{\l^2_n}^{2\ell}.
\end{align}
Thus, on any time interval where 
\begin{align}\label{8:54}
\nm{\ah(t)}_{\l^2_n}^2\leq \tfrac1{20},
\end{align}
the series on the right-hand side of \eqref{8:52} converges and satisfies
\begin{align*}
\Bigl| \bigl| M(\alpha(t))\bigr|  - \nm{\ah(t)}_{\l_n^2}^2 \Bigr| \leq \tfrac{1}{20}\nm{\ah(t)}_{\l^2_n}^2
\end{align*}
and so,
\begin{align}\label{8:53}
\tfrac12 \bigl| M(\alpha(t))\bigr| \leq \nm{\ah(t)}_{\l_n^2}^2 \leq 2 \bigl| M(\alpha(t))\bigr|.
\end{align}

The restriction $h\leq h_0$ together with \eqref{alpha0 mass} ensures that 
\begin{align}\label{1:49}
\|\alpha(0)\|_{\ell^2_n}^2 \leq \tfrac1{100}.
\end{align}
Using this, \eqref{8:53}, and the conservation of $M(\alpha(t))$, a simple bootstrap argument shows that the solution is global and that \eqref{8:54} and \eqref{8:53} hold for all $t\in\R$.
\end{proof}

Combining Proposition~\ref{Prop:l^2-bdd} with the conservation of the energy $E(\ah)$, we will show that the family of functions $\ph(t,x)$ defined in \eqref{psi-def} is bounded in $H^1_x(\R)$, uniformly for $t\in\R$ and $0<h\leq h_0$.

\begin{proposition}[Uniform bounds in $H_x^1$] \label{Prop:H^1-bdd}
Fix $0<h\leq h_0:=\min\big\{1,\frac{1}{100\|\phi_0\|_{2}^2}\big\}$ and let $\alpha\in C_t\ell_n^2$ denote the global solution to \eqref{mAL} with initial data \eqref{E:initial data} guaranteed by Proposition~\ref{Prop:l^2-bdd}.  Then
\begin{align} \label{ah-H^1}
\sup_{t\in\r}\bnm{\!\abs{\theta}\wah(t,\theta)}_{L^2_{\theta}}\ls h^{\frac32}\bigl[ \|\phi_0\|_{H^1}+ \|\phi_0\|_{L^2}^3\bigr],
\end{align}
with the implicit constant independent of $h$.  Moreover,
\begin{align}\label{unif bdd}
\sup_{0<h\leq h_0}\big\|\phi^h\big\|_{L^\infty_t H^1_x(\R\times\R)} \ls  \|\phi_0\|_{H^1}+ \|\phi_0\|_{L^2}^3.
\end{align}
\end{proposition}

\begin{proof}
A straightforward computation reveals that the  quadratic part of the conserved quantity $E(\ah)$ is given by
\begin{align} \label{E^2}
    E^{[2]}\big(\ah(t)\big) 
    &= \sum_{n\in\Z}\Big\{ 2\ah_n(t)\bh_n(t) -\ah_n(t)\bh_{n+1}(t) -\ah_{n+1}(t)\bh_n(t) \Big\}\no \\
    &= \pm\sum_{n\in\Z} \babs{\ah_{n+1}(t)-\ah_n(t)}^2\no \\
    &= \pm\int_{-\pi}^{\pi} 4\sin^2\!\left(\tfrac{\theta}{2}\right)\babs{\wah(t,\theta)}^2 \tfrac{d\theta}{2\pi}
    \simeq\pm \bnm{\!\abs{\theta}\wah(t,\theta)}_{L^2_{\theta}}^2,
\end{align}
where we used that $\frac{2\sin\left(\frac{\theta}{2}\right)}{\theta}\in\left[\tfrac{2}{\pi},1\right]$ for $\abs{\theta}\leq\pi$.  Moreover, using \eqref{E:mass bound} and our assumption on $h$, the higher-order terms may be bounded as follows:
\begin{align} \label{E-higher}
\Big|E\big(\ah(t)\big) - E^{[2]}\big(\ah(t)\big) \Big| 
&= 2\Big|\sum_{n\in\Z}\sum_{k\geq2}\tfrac{1}{k}\big[\ah_n(t)\bh_n(t)\big]^k \Big| \no\\
&\leq \bnm{\ah(t)}_{\l_n^4}^4 + 2\sum_{k\geq3}\tfrac{1}{k} \bnm{\ah(t)}_{\l_n^2}^{2k} \no\\
&\ls \bnm{\ah(t)}_{\l_n^4}^4+ \bnm{\ah(t)}_{\l_n^2}^6
 \ls \bnm{\ah(t)}_{\l_n^4}^4 + h^3 \|\phi_0\|_{L^2}^6 .
\end{align}

Using the Hausdorff--Young inequality and decomposing into frequencies $\abs{\theta}\leq A$ and $\abs{\theta}>A$ (for some $0<A\leq\pi$ to be determined later), we may bound
\begin{align*}
    \nm{\ah}_{\l_n^4} \leq \nm{\wah}_{L^{\frac{4}{3}}_{\theta}}
    &\leq \nm{\id_{\abs{\theta}\leq A}\,\wah}_{L^{\frac{4}{3}}_{\theta}}+\nm{\id_{\abs{\theta}\geq A}\,\wah}_{L^{\frac{4}{3}}_{\theta}} \\
    &\leq \nm{\wah}_{L^2_{\theta}}\nm{\id_{\abs{\theta}\leq A}}_{L^4_{\theta}} +\bnm{\!\abs{\theta}\wah}_{L^2_{\theta}}\bnm{\id_{\abs{\theta}\geq A}\tfrac{1}{\abs{\theta}}}_{L^4_{\theta}} \\
    &\ls A^{\frac{1}{4}}\nm{\wah}_{L^2_{\theta}} 
    +A^{-\frac{3}{4}}\bnm{\!\abs{\theta}\wah}_{L^2_{\theta}},
\end{align*}
where all $L^p_\theta$ norms are over $[-\pi,\pi]$.  We choose $A=\nm{\wah}_{L^2}^{-1}\bnm{\!\abs{\theta}\wah}_{L^2}$, which satisfies $0<A\leq\pi$, to optimize the bound above. Together with \eqref{E:mass bound}, this gives
\begin{align} \label{ah-l^4}
    \bnm{\ah(t)}_{\l_n^4}^4 \ls \bnm{\wah(t)}_{L^2_{\theta}}^{3} \bnm{\!\abs{\theta}\wah(t,\theta)}_{L^2_{\theta}} \ls \ep \bnm{\!\abs{\theta}\wah(t,\theta)}_{L^2_{\theta}}^2 + \ep^{-1} h^3 \|\phi_0\|_{L^2}^6
\end{align}
uniformly for $0<\ep\leq1$.

Combining \eqref{E^2}, \eqref{E-higher}, and \eqref{ah-l^4} with the conservation of $E(\ah)$ and taking $\ep$ small to defeat the implicit constants, we deduce that
\begin{align*}      
\bnm{\!\abs{\theta}\wah(t,\theta)}_{L^2_{\theta}}^2 \ls
\bnm{\!\abs{\theta}\wah(0,\theta)}_{L^2_{\theta}}^2 + h^3 \|\phi_0\|_{L^2}^6.
\end{align*}
Recalling \eqref{E:initial data} and \eqref{a0hat}, we find
\begin{align*}
    \bnm{\!\abs{\theta}\wah(0,\theta)}_{L^2_{\theta}}^2
    = \!\!\int_{-\pi}^{\pi} \abs{\theta}^2 \babs{\widehat{P_{\leq \frac{\pi}{2h}} \pe_0}(h^{-1}\theta)}^2 \tfrac{\ud\theta}{2\pi}
    = h^3\!\! \int_{-\frac\pi h}^{\frac\pi h} \abs{\xi}^2 \babs{\widehat{\pe_0}(\xi)}^2 \tfrac{\ud\xi}{2\pi}
    \leq h^3 \|\phi_0\|_{\dot{H}^1}^2,
\end{align*}
thus completing the proof of \eqref{ah-H^1}.

We now turn to \eqref{unif bdd}.  From \eqref{phi^h-hat}, we find
\begin{align} \label{phi^h-H^1}
\bnm{\ph(t)}_{H^1_x}^2 = \int_{\R} \big(1+|\xi|^2\big) \babs{\widehat{\ph}(t,\xi)}^2 \,\tfrac{d\xi}{2\pi}
    &= \tfrac1h\!\int_{-\pi}^{\pi}\big(1+h^{-2}\!\abs{\theta}^2\big)\abs{\wah(3h^{-3}t,\theta)}^2\tfrac{d\theta}{2\pi} \no\\
    &= \tfrac1h \nm{\wah(3h^{-3}t)}_{L^2_{\theta}}^2+\tfrac1{h^3}\nm{\abs{\theta}\wah(3h^{-3}t,\theta)}_{L^2_{\theta}}^2.
\end{align}
By Plancherel and Proposition~\ref{Prop:l^2-bdd}, the first term above is bounded by 
\begin{align}
    h^{-1} \nm{\ah(3h^{-3}t)}_{\l^2_n}^2 \ls \|\phi_0\|_{L^2}^2.
\end{align}
Using \eqref{ah-H^1} to bound the second term in \eqref{phi^h-H^1}, we deduce \eqref{unif bdd}.
\end{proof}

Combining Proposition~\ref{Prop:H^1-bdd} with the embedding $H^1(\R)\hookrightarrow L^\infty(\R)$ and \eqref{aa 01}, we obtain  

\begin{corollary}[$L^\infty_x$/$\l^\infty_n$-bound] \label{Cor:L^infty}
We have
\begin{align*}
\big\|\phi^h(t)\big\|_{L^\infty_{x}} \ls \|\phi_0\|_{H^1} +\|\phi_0\|_{L^2}^3
    \qquad\text{and}\qquad
    \bnm{\ah(t)}_{\l_n^\infty}\ls h  \bigl[\|\phi_0\|_{H^1}+\|\phi_0\|_{L^2}^3\bigr],
\end{align*}
with the implicit constants independent of $t$ and $h$.
\end{corollary}

\section{Precompactness in \texorpdfstring{$C_tL_x^2$}{}}\label{S:4}

This section is dedicated to the proof of the following precompactness result:

\begin{proposition}[Precompactness in $C_tL_x^2$]\label{thm:pre-compactness}
For $T>0$ fixed,
    the family of functions $\big\{\ph:[-T,T]\times\R\to \C \, \big| \,  0<h\leq h_0\big\}$ is precompact in $C([-T,T];L^2_x(\R))$.
\end{proposition}

By the Arzel\`a--Ascoli theorem and its $L^2$ analogue, due to M. Riesz \cite{Riesz1933}, precompactness of the family is equivalent to the following three properties:

\noindent\emph{Uniform boundedness}: There exists $C>0$ such that
\begin{align}\label{unif bdd=}
\sup_{0<h\leq h_0}\bnm{\ph}_{L^\infty_t L^2_x([-T,T]\times\R)}\leq C.
\end{align}
\noindent\emph{Equicontinuity}: For any $\ep>0$, there exists $\delta>0$ so that whenever $\abs{s}+\abs{y}<\delta$,
\begin{align} \label{equi}
\nm{\ph(t+s,x+y)-\ph(t,x)}_{L^2_x}< \ep
\end{align}
uniformly for $t\in[-T,T]$ (with $t+s\in [-T,T]$) and $0<h\leq h_0$.\\[2mm]
\noindent\emph{Tightness}: For any $\ep>0$, there exists $R>0$ such that
\begin{align} \label{tightness}
\sup_{0<h\leq h_0}  \sup_{|t|\leq T}\int_{\abs{x}\geq R}\babs{\ph(t,x)}^2\, dx<\ep.
\end{align}

The uniform boundedness property \eqref{unif bdd=} is guaranteed by Proposition~\ref{Prop:H^1-bdd}.  Next, we will demonstrate the equicontinuity statement \eqref{equi} by treating separately the space and time variables. Finally, we will demonstrate the tightness property in subsection~4.3.

\subsection{Equicontinuity in Space}
By Plancherel,
\begin{align*} 
\bigl\|\phi^h(t,x+y)-\phi^h(t,x)\bigr\|_{L^2_x}^2
&= \int_{\R} \babs{\ue^{\ui y\xi}-1}^2\babs{\wph(t,\xi)}^2 \tfrac{\ud\xi}{2\pi} \\
&= \int_{\R} 4\sin^2\!\big(\tfrac{y\xi}{2}\big) \babs{\wph(t,\xi)}^2 \tfrac{\ud\xi}{2\pi}.
\end{align*}
For $\kappa\geq 1$ to be chosen shortly, we may bound
\begin{align*}
\int_{\R} 4\sin^2\!\big(\tfrac{y\xi}{2}\big) \babs{\wph(t,\xi)}^2 \tfrac{\ud\xi}{2\pi}
&\lesssim \int_{|\xi|\leq \kappa} |y|^2|\xi|^2\babs{\wph(t,\xi)}^2 \, d\xi +  \kappa^{-2} \int_{|\xi|>\kappa} \langle \xi\rangle^2 \babs{\wph(t,\xi)}^2 \, d\xi\\
&\lesssim |y|^2\kappa^2 \| \ph(t)\|_{L^2_x}^2+ \kappa^{-2} \| \ph(t)\|_{H^1_x}^2.
\end{align*}
Using \eqref{unif bdd}, for any $\ep>0$ we may first choose $\kappa$ large and then $\delta$ small so that $|y|<\delta$ guarantees that
$$
\bigl\|\phi^h(t,x+y)-\phi^h(t,x)\bigr\|_{L^2_x}<\ep,
$$
uniformly for $t\in[-T,T]$ and $0<h\leq h_0$.

\subsection{Equicontinuity in Time}
By Plancherel and \eqref{phi^h-hat},
\begin{align*}
\bigl\|\phi^h(t+s,x)-\phi^h(t,x)\bigr\|_{L^2_x}
= h^{-\frac12} \nm{\ue^{6\ui h^{-3} s\theta }\,\wah\big(3h^{-3}(t+s),\theta\big)-\wah\big(3h^{-3}t,\theta\big)}_{L^2_\theta}.
\end{align*}

By working in Fourier variables, the Duhamel formula associated with \eqref{mAL} yields
\begin{align*}
\wah\big(3h^{-3}(t+s),\theta\big) &= \ue^{-6\ui h^{-3}s\sin(\theta)}\wah\big(3h^{-3}t,\theta\big)\\
&\quad +\int_{3h^{-3}t}^{3h^{-3}(t+s)}\ue^{-2\ui \sin(\theta)[3h^{-3}(t+s)-\tau]}\widehat{F}(\tau,\theta)\,\ud\tau,
\end{align*}
where $F_n:= \ah_n\bh_n\big(\ah_{n+1}-\ah_{n-1}\big)$.  Using Lemma~\ref{L:sums to int}, we get
\begin{align} \label{temp-1}
h^{\frac12}\bigl\|\phi^h(t+s,x)-\phi^h(t,x)\bigr\|_{L^2_x}
& \leq \nm{\big[\ue^{6\ui h^{-3} s(\theta-\sin\theta) }-1\big] \wah\big(3h^{-3}t,\theta\big) }_{L^2_\theta} \no\\
&\quad+ \Bnm{\int_{3h^{-3}t}^{3h^{-3}(t+s)}\ue^{-2\ui \sin(\theta)[3h^{-3}(t+s)-\tau]}\widehat{F}(\tau,\theta)\,\ud\tau }_{L^2_\theta}. 
\end{align}

We first consider the contribution of the linear term.  For $\kappa\geq 1$ to be chosen shortly, we evaluate the contributions of the regions $|\theta|< \kappa h$ and $\kappa h\leq|\theta|\leq\pi$ separately, using Proposition~\ref{Prop:l^2-bdd} and \eqref{ah-H^1}:
\begin{align*}
\nm{\big[\ue^{6\ui h^{-3} s(\theta-\sin\theta) }-1\big] \wah\big(3h^{-3}t,\theta\big) }_{L^2_\theta}
&\ls \kappa^3|s| \bigl\|\wah(3h^{-3}t)\bigr\|_{L_\theta^2}+ \tfrac1{kh}\bigl\||\theta| \wah(3h^{-3}t,\theta)\bigr\|_{L_\theta^2} \\
&\ls  \kappa^3|s| h^{\frac12}\|\phi_0\|_{L^2}+ \tfrac{h^{\frac12}}\kappa\bigl[\|\phi_0\|_{H^1} + \|\phi_0\|_{L^2}^3 \bigr].
\end{align*}
Thus,  for any $\ep>0$ we may first choose $\kappa$ large and then $\delta$ small so that $|s|<\delta$ guarantees that 
\begin{align} \label{temp-lin}
\nm{\big[\ue^{6\ui h^{-3} s(\theta-\sin\theta) }-1\big] \wah\big(3h^{-3}t,\theta\big) }_{L^2_\theta}<\tfrac\ep2 h^{\frac12},
\end{align}
uniformly for $t\in[-T,T]$ and $0<h\leq h_0$.

To estimate the contribution of the nonlinearity, we use Plancherel, the Minkowski and H\"older inequalities, followed by \eqref{ah-H^1} and Corollary~\ref{Cor:L^infty}:
\begin{align*}
\Bnm{\int_{3h^{-3}t}^{3h^{-3}(t+s)}\ue^{-2\ui \sin(\theta)[3h^{-3}(t+s)-\tau]}&\widehat{F}(\tau,\theta)\,\ud\tau }_{L^2_\theta} \\
&\ls h^{-3}|s| \sup_{\tau}\bnm{F(\tau)}_{\ell_n^2} \\
&\ls h^{-3}|s| \sup_{\tau}\bigl[\bnm{\ah(\tau)}_{\ell_n^\infty}^2 \bnm{\sin(\theta)\wah(\tau,\theta)}_{L^2_\theta}\bigr] \\
&\ls h^{-3}|s| h^{\frac72}\bigl[ \|\phi_0\|_{H^1} + \|\phi_0\|_{L^2}^3\bigr]^3\\
&\leq \tfrac\ep2 h^{\frac 12} 
\end{align*}
for $|s|<\delta=\delta(\ep)$ sufficiently small.  Combining this with \eqref{temp-1} and \eqref{temp-lin} yields
$$
\nm{\ph(t+s)-\ph(t)}_{L^2_x}< \ep \quad\text{whenever} \quad |s|<\delta,
$$
uniformly for $t\in[-T,T]$ (with $t+s\in [-T,T]$) and $0<h\leq h_0$.

\subsection{Tightness}
In view of the well-posedness of \eqref{mAL} discussed in Proposition~\ref{Prop:l^2-bdd}, the map $h\mapsto\ph$ is continuous from $(0,h_0]$ to the space $C([-T,T];L^2_x(\R))$.  Thus,  property \eqref{tightness} automatically holds on any compact interval $[h_1,h_0]$, with $h_1>0$.  Correspondingly, it suffices to prove \eqref{tightness} only for very small $h\in(0,h_1]$ where $h_1$ may depend on $\ep$.

For $L\geq 1$ to be chosen later and $R\geq 2L$, we may use \eqref{psi-def-R}, Lemmas~\ref{L:R-opt} and \ref{lem:reconstruction2}, and Proposition~\ref{Prop:l^2-bdd} to bound
\begin{align} \label{tightness-transfer}
\int_{\abs{x}\geq R}\babs{\ph(t,x)}^2\, dx
&= \int_{\abs{x-6h^{-2}t}\geq R} \babs{[\rr\ah](3h^{-3}t,x)}^2\, dx \no\\
&\ls \Bnm{\rr\big[\ah\id_{\{|n-6h^{-3}t|> h^{-1}L\}}\big](3h^{-3}t)}_{L^2_x}^2 \no\\
&\quad + \int_{\abs{x-6h^{-2}t}\geq R} \Babs{\rr\big[\ah\id_{\{|n-6h^{-3}t|\leq h^{-1}L\}}\big](3h^{-3}t,x)}^2 dx \no\\
&\ls h^{-1}\Bnm{\big[\ah\id_{\{|n-6h^{-3}t|> h^{-1}L\}}\big](3h^{-3}t)}_{\l^2_n}^2+ h^{-1}\tfrac{L}{R-L} \bnm{\ah(3h^{-3}t)}_{\l^2_n}^2 \no\\
&\ls h^{-1}\Bnm{\big[\ah\id_{\{|n-6h^{-3}t|> h^{-1}L\}}\big](3h^{-3}t)}_{\l^2_n}^2 + \tfrac{L}{R-L}\|\phi_0\|_{L^2}^2.
\end{align}

Let $\chi(x)\in C^{\infty}(\R)$ be a smooth bump function satisfying $\chi(x)=1$ for $\abs{x}\leq1$ and $\chi(x)=0$ for $\abs{x}\geq2$.  From this we build a cutoff function to large $n$ on the lattice such that $|n-6h^{-3}t|> \frac{L}{2h}$ via  
\begin{align}
    w=w(h,L;\tau,n) := 1-\chi\big(\tfrac{2(n-2\tau)h}{L}\big)
\end{align}
where we use the shorthand $\tau=3h^{-3}t$.

We will show that for any $0<\ep\leq 1$ there exists $L\geq 1$ and $h_1(\ep)>0$ such that 
\begin{align} \label{tightness 2}
\sup_{|\tau|\leq 3h^{-3}T} \sum_{n\in\Z}w(h,L;\tau,n)\babs{\ah_n(\tau)}^2 <\ep h, 
\end{align}
uniformly for $0<h\leq h_1(\ep)$.  Choosing $R=\frac{2L}\ep$, \eqref{tightness-transfer} then guarantees that
\begin{align*}
\sup_{0<h\leq h_1}  \sup_{|t|\leq T}\int_{\abs{x}\geq R}\babs{\ph(t,x)}^2\, dx\ls\ep.
\end{align*}
Recall that tightness in the regime $h\in[h_1,h_0]$ follows from the compactness of the interval $[h_1,h_0]$ and the continuity of the mapping $h\mapsto\ph$.

It remains to prove \eqref{tightness 2}. Using \eqref{mAL}, we compute
\begin{align} \label{tightness'-dt}
\frac{\ud}{\ud\tau} \sum_{n\in\Z} w(h,L;\tau,n)\babs{\ah_n(\tau)}^2 
&= \sum_{n\in\Z} \left\{w(n)\cdot 2\Re\big[\overline{\ah}_n\partial_{\tau}\ah_n\big] + \partial_{\tau}w(n)\abs{\ah_n}^2 \right\} \no\\
&= 2\Re\sum_{n\in\Z} w(n)\cdot\overline{\ah}_n\Big\{-\big(\ah_{n+1}-\ah_{n-1}\big) + F_n\Big\} \no\\
&\quad + \tfrac{4h}{L} \sum_{n\in\Z}\chi'(x_n)\abs{\ah_n}^2 ,
\end{align}
where $F_n:= \ah_n\bh_n\big(\ah_{n+1}-\ah_{n-1}\big)$ and we use the shorthand $x_n:=\tfrac{2(n-2\tau)h}{L}$.

To exploit cancellations between the quadratic terms above, we rewrite
\begin{align*}
-2\Re\sum_{n\in\Z} w(n)\cdot\overline{\ah}_n\big(\ah_{n+1}-\ah_{n-1}\big)
&= 2\Re\sum_{n\in\Z} \big[w(n+1)-w(n)\big] \overline{\ah}_n\ah_{n+1} \\
&= \sum_{n\in\Z} \big[w(n+1)-w(n-1)\big] \abs{\ah_n}^2 \\
&\quad - \sum_{n\in\Z} \big[w(n+1)-w(n)\big] \babs{\ah_{n+1}-\ah_n}^2.
\end{align*}
Using the Taylor expansion (with Lagrange remainder), we may write
\begin{align*}
w(n+1)-w(n) &= -\tfrac{2h}{L}\chi'(\xi_n),\\
w(n+1)-w(n-1) &= -\tfrac{4h}{L}\chi'(x_n) - \tfrac{1}{6}\big(\tfrac{2h}{L}\big)^3 \big[\chi'''(\theta_n)+\chi'''(\eta_n)\big],
\end{align*}
for some $\xi_n, \theta_n\in(x_n,x_{n+1})$ and $\eta_n\in(x_{n-1},x_{n})$.  Thus, cancelling with the last term in \eqref{tightness'-dt} and using Proposition~\ref{Prop:l^2-bdd}, we may estimate
\begin{align*}
\bigg|\sum_{n\in\Z} \big[w(n+1)-w(n-1)\big] \abs{\ah_n}^2 + \tfrac{4h}{L} \sum_{n\in\Z}\chi'(x_n)\abs{\ah_n}^2\bigg|
\ls \big(\tfrac{h}{L}\big)^3 \nm{\ah}_{\l_n^2}^2
\ls \tfrac{h^4}{L^3}\|\phi_0\|_{L^2}^2.
\end{align*}
Moreover, using Proposition~\ref{Prop:H^1-bdd} we may bound
\begin{align*}
\bigg|\sum_{n\in\Z} \big[w(n+1)-w(n)\big] \babs{\ah_{n+1}-\ah_n}^2 \bigg| 
&\ls \bnm{w(n+1)-w(n)}_{\l_n^\infty} \bnm{\ah_{n+1}-\ah_n}_{\l_n^2}^2 \\
&\ls \tfrac{h}{L} \bnm{\!\abs{\theta}\wah(\tau,\theta)}_{L^2_{\theta}}^2 \\
&\ls \tfrac{h^4}{L}\bigl[ \|\phi_0\|_{H^1}^2 + \|\phi_0\|_{L^2}^6\bigr].
\end{align*}
Thus, we may bound the contribution of the quadratic terms in \eqref{tightness'-dt} by
\begin{align} \label{dt-lin}
\bigg|-2\Re\sum_{n\in\Z} w(n)\cdot\overline{\ah}_n\big(\ah_{n+1}-\ah_{n-1}\big) &+ \tfrac{4h}{L} \sum_{n\in\Z}\chi'(x_n)\abs{\ah_n}^2 \bigg| \no\\
&\ls \tfrac{h^4}{L} \bigl[ \|\phi_0\|_{H^1}^2 + \|\phi_0\|_{L^2}^6\bigr].
\end{align}

Next we turn to the contribution of the nonlinearity in \eqref{tightness'-dt},
which can be rewritten as 
\begin{align*}
2\Re\sum_{n\in\Z} &w(n)\cdot\overline{\ah}_n F_n \\
    &= 2\Re\sum_{n\in\Z} w(n)\cdot \abs{\ah_n}^2\bh_n\big(\ah_{n+1}-\ah_{n-1}\big) \\
    &=\sum_{n\in\Z} \Big[w(n-1)\abs{\ah_{n-1}}^2-w(n+1)\abs{\ah_{n+1}}^2\Big] \ah_n\bh_n\\ 
    &\quad- \sum_{n\in\Z} \Big[w(n)\abs{\ah_{n}}^2-w(n+1)\abs{\ah_{n+1}}^2\Big] \big(\ah_{n+1}-\ah_n\big)\big(\bh_{n+1}-\bh_n\big) \\
    &= \pm\sum_{n\in\Z} \big[w(n)-w(n+1)\big] \abs{\ah_n}^2\abs{\ah_{n+1}}^2 \\
    &\quad \mp \sum_{n\in\Z} \Big[w(n)\abs{\ah_{n}}^2-w(n+1)\abs{\ah_{n+1}}^2\Big] \babs{\ah_{n+1}-\ah_n}^2.
\end{align*}
Using Proposition~\ref{Prop:l^2-bdd} and Corollary~\ref{Cor:L^infty}, we may bound the contribution of the first series above by
\begin{align*}
\bnm{w(n+1)-w(n)}_{\l_n^\infty} \nm{\ah}_{\l_n^\infty}^2 \nm{\ah}_{\l_n^2}^2 \ls \tfrac{h^4}{L} \bigl[\|\phi_0\|_{H^1}^2+ \|\phi_0\|_{L^2}^6\bigr]\|\phi_0\|_{L^2}^2.
\end{align*}
Similarly, using also Proposition~\ref{Prop:H^1-bdd}, the contribution of the second series can be bounded by 
\begin{align*}
\nm{w}_{\l_n^\infty} \nm{\ah}_{\l_n^\infty}^2 \bnm{\ah_{n+1}-\ah_n}_{\l_n^2}^2 \ls h^5 \bigl[\|\phi_0\|_{H^1}^2+ \|\phi_0\|_{L^2}^6\bigr]^2.
\end{align*}
Hence, we may bound the contribution of the nonlinearity by
\begin{align} \label{dt-non}
    \bigg|2\Re\sum_{n\in\Z} w(n)\cdot\overline{\ah}_n F_n\bigg|\ls \bigl(\tfrac{h^4}{L} + h^5\bigr) \bigl[\|\phi_0\|_{H^1}^2+ \|\phi_0\|_{L^2}^6\bigr]^2.
\end{align}

Combining \eqref{tightness'-dt} with \eqref{dt-lin} and \eqref{dt-non}, we find 
\begin{align*}
\bigg|\frac{\ud}{\ud\tau} \sum_{n\in\Z} w(h,L;\tau,n)\babs{\ah_n(\tau)}^2\bigg|\ls \tfrac{h^4}{L} + h^5 \quad\text{uniformly for $\tau\in \R$,}
\end{align*}
with the implicit constant depending only on the $H^1$ norm of $\phi_0$.  Integrating in the time variable, this yields
\begin{align} \label{tightness 2'}
\sup_{|\tau|\leq 3h^{-3}T} \sum_{n\in\Z}w(h,L;\tau,n)\babs{\ah_n(\tau)}^2 &\ls \sum_{n\in\Z} w(h,L;0,n) \babs{\ah_n(0)}^2
+ 3h^{-3}T \big(\tfrac{h^4}{L} + h^5\big) \no \\
&\ls \sum_{n\in\Z} \big[1-\chi\big(\tfrac{2nh}{L}\big)\big] \babs{\ah_n(0)}^2
+ T \big(\tfrac{1}{L} + h\big)h.
\end{align}
Using Lemma~\ref{L:locA}, we find
\begin{align*}
h^{-1} \sum_{n\in\Z} \big[1-\chi\big(\tfrac{2nh}{L}\big)\big] \babs{\ah_n(0)}^2 
    \ls  \bnm{\mathcal{M} \phi_0 }_{L^2_x(|x|\geq L)}.
\end{align*}
By the dominated convergence theorem and the $L^2$-boundedness of the Hardy--Littlewood maximal function, the right-hand side above converges to zero as $L\to \infty$.  Thus, the right-hand side of \eqref{tightness 2'} can be made smaller than $\ep h$ by choosing $L=L(\ep)$ sufficiently large and then $0<h\leq h_1(\ep)$ sufficiently small.
This completes the proof of \eqref{tightness 2} and so that of \eqref{tightness}.

\section{Convergence of the Flows} \label{Sec:Convergence}

It follows from Proposition~\ref{thm:pre-compactness} that every sequence $h_n\rt 0$ admits a subsequence along which $\pe^{h_{n}}$ converges to some $\pe$ in $C([-T,T];L^2_x(\R))$.   In this section, we will show that all such subsequential limits are $L^\infty_tH_x^1$ solutions to the Duhamel formulation of \eqref{mKdV} with initial data $\pe_0$.  Using the uniqueness of such \eqref{mKdV} solutions provided by Theorem~\ref{T:UCU}, we conclude that all subsequential limits agree and so $\phi^h \to \phi$  in $C([-T,T];L^2_x(\R))$ as $h\to 0$.

To better appreciate the connection between \eqref{mKdV} and \eqref{mAL}, we rewrite \eqref{mKdV} as
\begin{align}
    \dt\pe=-\p_x^3\pe\pm6\abs{\pe}^2\p_x\pe
    =: \ui\Lc\pe + F[\pe]
\end{align}
and \eqref{mAL} as 
\begin{align}
\dt\ah_n = -\big(\ah_{n+1}-\ah_{n-1}\big) + \ah_n\bh_n\big(\ah_{n+1}-\ah_{n-1}\big)
=: \ui(\Ld\ah)_n + F_n[\ah], 
\end{align}
where $\Lc, \Ld$ are differential/difference operators with symbols
\begin{align*}
\Lc(\xi)=\xi^3 \quad\text{and}\quad \Ld(\theta)=-2\sin(\theta),
\end{align*}
while the nonlinearities are given by 
\begin{align*}
F[\pe]=\pm6\abs{\pe}^2\p_x\pe \quad\text{and} \quad F_n[\ah]=\ah_n\bh_n\big(\ah_{n+1}-\ah_{n-1}\big).
\end{align*}

\begin{proposition}\label{Prop:convergence}
Let $\phi \in C([-T,T];L^2_x(\R))$ be such that 
\begin{align}\label{convg}
\phi^{h_n} \to\phi \qtq{in} C([-T,T];L^2_x(\R))
\end{align}
along some sequence $h_n\to 0$.  Then $\pe \in L^\infty_t H^1_x([-T,T]\times\R)$ and it solves \eqref{mKdV} with initial data $\pe_0$ in the sense that 
\begin{align}\label{duhamel}
    \pe(t)=\ue^{\ui t\Lc}\pe_0 + \int_0^t\ue^{\ui(t-s)\Lc}F[\pe(s)]\,\ud s
\end{align}
in $C([-T,T];L^2_x(\R))$.
\end{proposition}

\begin{proof}
For notational simplicity, we will omit the subscript on $h$ in what follows.

As $\phi^h$ converges to $\pe$ in $C([-T,T];L^2_x(\R))$, \eqref{unif bdd} and weak lower-semicontinuity of the $H^1$ norm imply 
\begin{align} \label{phi-H^1}
\|\phi\|_{L^\infty_t H^1_x([-T,T]\times\R)} \ls \|\phi_0\|_{H^1} + \|\phi_0\|_{L^2}^3.
\end{align}
Consequently, $\phi^h$ converges to $\pe$ in $C([-T,T]; H^s_x(\R))$ for all $0\leq s<1$.  Indeed, by \eqref{unif bdd} and \eqref{phi-H^1},
\begin{align} \label{convg_H^s}
\bnm{\phi^{h}(t)-\phi(t)}_{H^s_x}
&\leq \bnm{\phi^{h}(t)-\phi(t)}_{H^1_x}^{s}\bnm{\phi^{h}(t)-\phi(t)}_{L^2_x}^{1-s} \no\\
&\ls \Big(\bnm{\phi^{h}(t)}_{H^1_x}^{s}+\nm{\phi(t)}_{H^1_x}^{s}\Big)\bnm{\phi^{h}(t)-\phi(t)}_{L^2_x}^{1-s}\no\\
&\ls \bigl[\|\phi_0\|_{H^1}+ \|\phi_0\|_{L^2}^3\bigr]^s \bnm{\phi^{h}(t)-\phi(t)}_{L^2_x}^{1-s} \to 0\quad\text{as} \quad h\to 0,
\end{align}
uniformly for $|t|\leq T$. 

In order to prove that the limit $\phi$ satisfies \eqref{duhamel}, our starting point is the Duhamel formula satisfied by the solution $\ah$ of \eqref{mAL}: For any $|t|\leq T$,
\begin{align}\label{ah duh}
    \ah_n(3h^{-3}t) = 
    e^{\ui 3h^{-3}t\Ld}\ah_n(0) 
    +3h^{-3}\int_0^{t} e^{\ui 3h^{-3}(t-s)\Ld}F_n\bigl[\ah(3h^{-3}s)\bigr]\,ds.
\end{align}
Recalling the relation \eqref{psi-def-R}, we find
\begin{equation}\label{ph duh}
\begin{aligned}
\ph(t) &= (\tt_t\circ\rr)\bigl[e^{\ui 3h^{-3}t\Ld}\ah_n(0)\bigr] \\
&\qquad + (\tt_t\circ\rr)\,\Bigl\{3h^{-3}\!\!\int_0^{t} \!\!e^{\ui 3h^{-3}(t-s)\Ld}F_n\bigl[\ah(3h^{-3}s)\bigr]\,ds\Bigr\} .
\end{aligned}
\end{equation}
We will prove \eqref{duhamel} by taking the limit as $h\to0$ of each term appearing in \eqref{ph duh}.  Evidently, \eqref{convg} ensures convergence of the left-hand side.  Convergence of each term on the right-hand side is the subject of our next two lemmas.

\begin{lemma}\label{L:convergence-linear}
Under the hypotheses of Proposition~\ref{Prop:convergence}, we have
\begin{align*}
\lim_{h\to 0}\,\Bigl\| (\tt_t\circ\rr)\bigl[e^{\ui 3h^{-3}t\Ld}\ah_n(0)\bigr] - e^{\ui t\Lc}\pe_0   \Bigr\|_{C_t L_x^2([-T,T]\times\R)} = 0.
\end{align*}
\end{lemma}

\begin{proof}
Using Plancherel, \eqref{E:initial data}, and \eqref{a0hat}, we have
\begin{align*}
\Bigl\| (\tt_t&\circ\rr)\bigl[ e^{\ui 3h^{-3}t\Ld}\ah_n(0)\bigr] - e^{\ui t\Lc}\pe_0   \Bigr\|_{C_t L_x^2([-T,T]\times\R)} \\
&\simeq\Bigl\| e^{\ui6h^{-2}t\xi} e^{\ui 3h^{-3}t[-2\sin(h\xi)]}\widehat\ah(0,h\xi) - e^{\ui t\xi^3}\widehat{\pe_0}(\xi)   \Bigr\|_{C_t L_\xi^2([-T,T]\times\R)}\\
&\leq\Bigl\| \Big[e^{\ui6h^{-3}t[h\xi-\sin(h\xi) -\frac{(h\xi)^3}{6}]} -1\Big] \widehat{P_{> \frac\pi{2h}}\pe_0}(\xi) \Bigr\|_{C_t L_\xi^2([-T,T]\times\R)}
    + \bigl\| \widehat{P_{\geq \frac{\pi}{2h}} \pe_0} \bigr\|_{L_\xi^2(\R)}.
\end{align*}
The claim now follows from the dominated convergence theorem. 
\end{proof}

\begin{lemma}\label{L:last goal} Under the hypotheses of Proposition~\ref{Prop:convergence}, we have
\begin{align}\label{last goal}
\lim_{h\to 0}\, (\tt_t\circ\rr)\,
\Bigl\{3h^{-3}\!\!\int_0^{t} \!\!e^{\ui 3h^{-3}(t-s)\Ld}F_n\bigl[\ah(3h^{-3}s)\bigr]\,ds\Bigr\} = \!\!\int_0^t\!\! \ue^{\ui(t-s)\Lc}F[\pe(s)]\,\ud s
\end{align}
in $C([-T,T];L^2_x(\R))$ sense.
\end{lemma}

Employing the sharp Fourier cutoff $P_h$ to $|\xi|< h^{-1}$, we define
\begin{align*}
\widetilde\phi^h := P_{h\,} \phi^h, \quad \widetilde \ah_n(3h^{-3}t) := h\widetilde\phi^h(t,nh-6h^{-2}t),
\qtq{and} \widetilde F_n(t) &:= F_n\bigl[\widetilde\ah(3h^{-3}t)\bigr].
\end{align*}
For later use, we note that by \eqref{unif bdd}, we have
\begin{align} \label{equi'_H^s}
\big\|\widetilde\phi^h - \phi^h \big\|_{L^\infty_t H_x^s} 
&\ls h^{1-s} \|\phi^h\|_{L^\infty_t H^1_x} \lesssim h^{1-s} [\|\phi_0\|_{H^1} + \|\phi_0\|_{L^2}^3] ,
\end{align}
for any $0\leq s\leq 1$ and $0<h\leq h_0$.  Our reason for projecting to this narrower frequency band is to allow us to apply Lemma~\ref{L:abc}.  Specifically, we have
\begin{align}\label{1065}
(\tt_s\circ\rr) \bigl[ \widetilde F_n(s) \bigr](x)  &= \pm h^2\big|\widetilde\phi^h(s,x)\big|^2\Big[\widetilde\phi^h(s,x+h)-\widetilde\phi^h(s,x-h)\Big] .
\end{align}

The next result shows that, for the purposes of Lemma~\ref{L:last goal}, the nonlinearity of \eqref{mAL} may be replaced by the one based on this more narrowly Fourier localized sequence.

\begin{lemma} \label{L:error}
Under the hypotheses of Proposition~\ref{Prop:convergence}, we have
\begin{align*}
\lim_{h\to0}\,\bigg\| (\tt_t\circ\rr)\,
\bigg\{\tfrac3{h^3}\!\int_0^{t} e^{\ui \frac{3(t-s)}{h^3}\Ld}\Big\{ F_n\bigl[\ah\bigl(\tfrac{3s}{h^3}\bigr)\bigr] -F_n\bigl[\widetilde\ah\bigl(\tfrac{3s}{h^3}\bigr)\bigr]\Big\} \,ds\bigg\}  \bigg\|_{C_t L_x^2([-T,T]\times\R)} =0.
\end{align*}
\end{lemma}

\begin{proof}
Using Lemma~\ref{L:R-opt} and performing a change of variables, we see that 
\begin{align*}
\bigg\| (\tt_t\circ\rr)\,&\bigg\{ \tfrac3{h^3}\!\int_0^{t}e^{\ui \frac{3(t-s)}{h^3}\Ld}\Big\{ F_n\bigl[\ah\bigl(\tfrac{3s}{h^3}\bigr)\bigr] -F_n\bigl[\widetilde\ah\bigl(\tfrac{3s}{h^3}\bigr)\bigr]\Big\} \,ds\bigg\}  \bigg\|_{C_t L_x^2([-T,T]\times\R)} \\
&\simeq h^{-\frac{1}{2}} \bigg\|\int_0^t e^{\ui (t-s)\Ld}\Big\{F_n\bigl[\ah(s)\bigr] -F_n\bigl[\widetilde\ah(s)\bigr]  \Big\} \,ds \bigg\|_{L_t^\infty \l_n^2([-3h^{-3}T,3h^{-3}T]\times\Z)}. 
\end{align*}

Using the discrete Strichartz inequality \eqref{Strichartz-ineq-d}, H\"older (for the time integral), and the discrete paraproduct estimate \eqref{paraproduct-2} from Lemma~\ref{L:paraproduct}, we may bound \begin{align*}
&h^{-\frac{1}{2}} \bigg\| \int_0^t  e^{\ui (t-s)\Ld}\Big\{F_n\bigl[\ah(s)\bigr] -F_n\bigl[\widetilde\ah(s)\bigr]  \Big\} \,ds \bigg\|_{L_t^\infty \l_n^2}\\
&\ls h^{-\frac{1}{2}} \Big\|\FDds^{-\frac{1}{8}} \Big\{ F_n\bigl[\ah(s)\bigr] -F_n\bigl[\widetilde\ah(s)\bigr]  \Big\} \Big\|_{L_t^{\frac{8}{7}}\l_n^{\frac{4}{3}}} \\
&\ls h^{-\frac{1}{2}} \big(\tfrac{3T}{h^{3}}\big)^{\frac{7}{8}}
    \bigg\{  \big\|\FDds^{-\frac{1}{8}} \big[\widetilde\alpha_{n+1}-\widetilde\alpha_{n-1}\big] \Big\|_{L_t^{\infty}\l_n^{2}} \cdot \bnm{\FDds^{\frac{3}{8}} \big[\alpha-\widetilde\alpha\big]}_{L_t^{\infty}\l_n^{2}} \\
&\qquad\qquad\qquad\qquad \qquad\cdot \Big[\bnm{\FDds^{\frac{3}{8}} \alpha}_{L_t^{\infty}\l_n^{2}} +\bnm{\FDds^{\frac{3}{8}} \widetilde\alpha}_{L_t^{\infty}\l_n^{2}} \Big]\\
&\qquad\qquad\qquad\quad +\big\|\FDds^{-\frac{1}{8}} \big[(\alpha - \widetilde\alpha)_{n+1}-(\alpha - \widetilde\alpha)_{n-1}\big] \big\|_{L_t^{\infty}\l_n^{2}} \bnm{\FDds^{\frac{3}{8}} \alpha}_{L_t^{\infty}\l_n^{2}}^2 \bigg\},
\end{align*}
where the all norms are over $[-3h^{-3}T,3h^{-3}T]\times\Z$.

Using \eqref{E:mass bound} and \eqref{ah-H^1}, we find
\begin{align*} 
\bnm{\FDds^{\frac{3}{8}} \alpha}_{L_t^{\infty}\l_n^{2}}  +\bnm{\FDds^{\frac{3}{8}} \widetilde\alpha}_{L_t^{\infty}\l_n^{2}}
&\lesssim h^{\frac78} \bigl[ \|\phi_0\|_{H^1} + \|\phi_0\|_{L^2}^3\bigr],\\
\big\|\FDds^{-\frac{1}{8}} \big[\widetilde\alpha_{n+1}-\widetilde\alpha_{n-1}\big] \big\|_{L_t^{\infty}\l_n^{2}}
\lesssim \big\|\FDds^{\frac{7}{8}}\widetilde\alpha\big\|_{L_t^{\infty}\l_n^{2}}
& \lesssim h^{\frac{11}8} \bigl[ \|\phi_0\|_{H^1} + \|\phi_0\|_{L^2}^3\bigr],\\
\bnm{\FDds^{\frac{3}{8}} \big[\alpha-\widetilde\alpha\big]}_{L_t^{\infty}\l_n^{2}}\lesssim \big\| |\theta|^{\frac38}\, \widehat\alpha(\theta)\big\|_{L_t^{\infty}L_\theta^{2}(1\leq|\theta|\leq \pi)}
&\lesssim  \big\||\theta| \,\widehat\alpha(\theta)\big\|_{L_t^{\infty}L_\theta^{2}}\\
& \lesssim h^{\frac32} \bigl[ \|\phi_0\|_{H^1} + \|\phi_0\|_{L^2}^3\bigr]
\end{align*}
and
\begin{align*}
\big\|\FDds^{-\frac{1}{8}} \big[(\alpha - \widetilde\alpha)_{n+1}-(\alpha - \widetilde\alpha)_{n-1}\big] \big\|_{L_t^{\infty}\l_n^{2}}
&\lesssim  \big\|\FDds^{\frac{7}{8}}(\alpha - \widetilde\alpha)\big\|_{L_t^{\infty}\l_n^{2}}\\
&\lesssim \big\| |\theta|^{\frac78}\, \widehat\alpha(\theta)\big\|_{L_t^{\infty}L_\theta^{2}(1\leq|\theta|\leq \pi)}\\
&\lesssim  \big\||\theta| \,\widehat\alpha(\theta)\big\|_{L_t^{\infty}L_\theta^{2}}\\
& \lesssim h^{\frac32} \bigl[ \|\phi_0\|_{H^1} + \|\phi_0\|_{L^2}^3\bigr].
\end{align*}

Collecting these bounds, we obtain
\begin{align*}
h^{-\frac{1}{2}} \bigg\|\int_0^t e^{\ui (t-s)\Ld}\Big\{F_n\bigl[\ah(s)\bigr] &-F_n\bigl[\widetilde\ah(s)\bigr]  \Big\} \,ds \bigg\|_{L_t^\infty \l_n^2([-\frac{3T}{h^{3}},\frac{3T}{h^{3}}]\times\Z)}\\
&\lesssim_T  h^{\frac18}\bigl[ \|\phi_0\|_{H^1} + \|\phi_0\|_{L^2}^3\bigr]^3,
\end{align*}
which converges to zero as $h\to 0$, thus completing the proof of the lemma.
\end{proof}

\begin{proof}[Proof of Lemma~\ref{L:last goal}]
As we have already demonstrated convergence of the other two terms in \eqref{ph duh}, the existence of the limit in \eqref{last goal} is already settled.  Our objective is to verify that the limit is the one given by the right-hand side of \eqref{last goal}.  In view of Lemma~\ref{L:error}, it suffices to show that
\begin{align}\label{10:03}
\Bigl\langle f,\, (\tt_t\circ\rr)\,
\Bigl\{\tfrac3{h^{3}} \! \int_0^{t} e^{\ui 3h^{-3}(t-s)\Ld}\widetilde{F}_n(s)\,ds\Bigr\} \Bigr\rangle
	\to \Bigl\langle f,\,\int_0^t\ue^{\ui(t-s)\Lc}F[\pe(s)]\,\ud s\Bigr\rangle 
\end{align}
as $h\to 0$, for any $f\in L^2(\R)$.

As a first step toward verifying \eqref{10:03}, we would like to pass from the discrete to the continuum propagator.  By the Plancherel Theorem,
\begin{align}
\Bigl\langle f, \ (\tt_t\circ\rr)\, &\Bigl\{ e^{\ui 3h^{-3}(t-s)\Ld}\widetilde{F}_n(s) \Bigr\} 
	- e^{\ui (t-s)\Lc} (\tt_s\circ\rr)\,\bigl\{\widetilde{F}_n(s)\bigr\}\Bigr\rangle_{\!L_x^2} \notag\\
&= \Bigl\langle \widehat f(\xi), \bigl[ e^{\ui 6 h^{-2}t\xi - 6\ui h^{-3}(t-s)\sin(h\xi)} - e^{\ui(t-s)\xi^3 + \ui 6 h^{-2}s\xi} \bigr]
	\widehat{\widetilde{F}}(s)\Bigr\rangle_{\!L_\xi^2} \notag\\
&= \Bigl\langle \bigl[ e^{-\ui (t-s)\omega_h(\xi)} - 1\bigr] \widehat f(\xi), e^{\ui(t-s)\xi^3 + \ui 6 h^{-2}s\xi} \widehat{\widetilde{F}}(s)\Bigr\rangle_{\!L_\xi^2} \label{1066}
\end{align}
where $\omega_h(\xi):=h^{-3}[6h\xi-6\sin(h\xi)-(h\xi)^3]$, which satisfies $|\omega_h(\xi)| \lesssim h^{-3}|h\xi|^5$.  Thus by the dominated convergence theorem,
\begin{align*}
\big\|[ e^{-\ui (t-s)\omega_h(\xi)} - 1\bigr] \widehat{f}(\xi) \big\|_{L_\xi^2} \to 0 \quad\text{uniformly for } s, t\in [-T, T],
\end{align*}
as $h\to 0$.  To treat the other side of the inner product in \eqref{1066}, we use Plancherel, \eqref{ah-H^1}, and Corollary~\ref{Cor:L^infty} to estimate
\begin{align*}
\bigl\| \widehat{\widetilde{F}}(s,h\xi) \bigr\|_{L_s^\infty L_\xi^2} 
\lesssim h^{-\frac12} \|\widetilde F\|_{L_t^\infty \ell^2_n}
&\lesssim h^{-\frac12} \|\alpha\|_{L_t^\infty \ell^\infty_n}^2 \|\alpha_{n+1} - \alpha_{n-1}\|_{L_t^\infty \ell^2_n}\\
&\lesssim h^3 \bigl[\nm{\phi_0}_{H^1}+ \|\phi_0\|_{L^2}^3\bigr]^3.
\end{align*}

Combining all the estimates of the previous paragraph, we deduce that
\begin{align}
\tfrac3{h^{3}} \! \int_0^{t}  \Bigl\langle f, (\tt_t\circ\rr)\, &\Bigl\{ e^{\ui 3h^{-3}(t-s)\Ld}\widetilde{F}_n(s) \Bigr\} 
	- e^{\ui (t-s)\Lc} (\tt_s\circ\rr)\,\bigl\{\widetilde{F}_n(s)\bigr\}\Bigr\rangle_{\!L_x^2}\,ds \to 0
\end{align}
as $h\to 0$, uniformly for $|t|\leq T$.  In this way, our goal of proving \eqref{10:03} may be fulfilled by showing that
\begin{align}\label{1068}
 \Bigl\|  \int_0^{t} \!\! e^{\ui (t-s)\Lc} \Bigl[ \pm6\abs{\pe(s)}^2\p_x\pe(s) - \tfrac3{h^{3}} (\tt_s\circ\rr)\,\bigl\{\widetilde{F}_n(s)\bigr\}\Bigr]\,ds  \Bigr\|_{L^2_x} \to 0
\end{align}
as $h\to 0$, uniformly for $|t|\leq T$.

We will prove \eqref{1068} by rewriting the nonlinearity using \eqref{1065} and then applying the Strichartz inequality \eqref{Strichartz-ineq}.  We estimate the resulting expression as follows, using 
the paraproduct estimate \eqref{paraproduct-1} from Lemma~\ref{L:paraproduct}, Proposition~\ref{Prop:H^1-bdd}, and \eqref{phi-H^1}:
\begin{align*}
\Big\| & \abs{\nabla}^{-\frac{1}{8}} \Big\{\big|\widetilde\phi^h\big|^2 \tfrac{1}{2h}\big[\widetilde\phi^h(\cdot+h)-\widetilde\phi^h(\cdot-h)\big] - \abs{\pe}^2\p_x\pe \Big\} \Big\|_{L_t^{\frac{8}{7}}L_x^{\frac{4}{3}}} \\
&\ls_T \Big\|\abs{\nabla}^{-\frac{1}{8}} \Big\{ \tfrac{1}{2h}\big[\widetilde\phi^h(\cdot+h)-\widetilde\phi^h(\cdot-h)\big] - \p_x\pe \Big\} \Big\|_{L_t^{\infty}L_x^{2}} \bnm{\widetilde\phi^h}_{L_t^{\infty}\dot{H}_x^{\frac{3}{8}}}^2 \notag\\
&\qquad + \big\|\abs{\nabla}^{-\frac{1}{8}} \p_x\pe \big\|_{L_t^{\infty}L_x^{2}} \bnm{\widetilde\phi^h -\pe}_{L_t^{\infty}\dot{H}_x^{\frac{3}{8}}}
    \Bigl[\bnm{\widetilde\phi^h}_{L_t^{\infty}\dot{H}_x^{\frac{3}{8}}} +\nm{\pe}_{L_t^{\infty}\dot{H}_x^{\frac{3}{8}}} \Bigr] \notag\\
&\ls_T \Bigl[\|\phi_0\|_{H^1}+ \|\phi_0\|_{L^2}^3\Bigr]^2 \notag\\
&\qquad \cdot\Bigl[\Big\|\abs{\nabla}^{-\frac{1}{8}} \Big\{ \tfrac{1}{2h}\big[\widetilde\phi^h(\cdot+h)-\widetilde\phi^h(\cdot-h)\big] - \p_x\pe \Big\} \Big\|_{L_t^{\infty}L_x^{2}} +  \bnm{\widetilde\phi^h -\pe}_{L_t^{\infty}\dot{H}_x^{\frac{3}{8}}}\Bigr],\notag
\end{align*}
where all spacetime norms are over the region $[-T,T]\times\R$. In view of \eqref{convg_H^s} and \eqref{equi'_H^s}, we have
\begin{align*}
\bnm{\widetilde\phi^h -\pe}_{L_t^{\infty}\dot{H}_x^{\frac{3}{8}}}
\leq \bnm{\phi^h -\pe}_{L_t^{\infty}\dot{H}_x^{\frac{3}{8}}} + \bnm{\phi^h-\widetilde\phi^h}_{L_t^{\infty}\dot{H}_x^{\frac{3}{8}}} \to 0  \quad\text{as $h\to 0$}.
\end{align*}
Moreover, noting that $\babs{\tfrac{\sin(h\xi)}{h\xi}-1} \ls \min\{1, h^2|\xi|^2\} \ls |h\xi|^{\frac18}$, we find that
\begin{align*}
\Big\|\abs{\nabla}^{-\frac{1}{8}} &\Big\{ \tfrac{1}{2h}\big[\widetilde\phi^h(\cdot+h)-\widetilde\phi^h(\cdot-h)\big] - \p_x\pe \Big\} \Big\|_{L_t^{\infty}L_x^{2}} \\
&\ls \Big\|\abs{\xi}^{\frac{7}{8}} \big[\tfrac{\sin(h\xi)}{h\xi} - 1 \big] \widehat{\widetilde\phi^h}(t,\xi) \Big\|_{L_t^{\infty}L_{\xi}^{2}}
    + \bnm{\widetilde\phi^h -\pe}_{L_t^{\infty}\dot{H}_x^{\frac{7}{8}}} \\
&\ls h^{\frac18} \big\| \widetilde\phi^h \big\|_{L_t^{\infty}H^1_x} 
    + \bnm{\widetilde\phi^h -\pe}_{L_t^{\infty}\dot{H}_x^{\frac{7}{8}}} .
\end{align*}
From \eqref{unif bdd} and \eqref{equi'_H^s}, we see that this converges to zero as $h\to0$.

Having shown that the difference of the nonlinearities from \eqref{1068} converges to zero in the space dictated by the Strichartz inequality \eqref{Strichartz-ineq}, we now know that \eqref{1068} holds.  This completes the proof of \eqref{10:03} and so that of Lemma~\ref{L:last goal}.
\end{proof}

Combining Lemmas~\ref{L:convergence-linear} and \ref{L:last goal} with \eqref{convg}, we see that \eqref{duhamel} follows from \eqref{ph duh}. The proof of Proposition~\ref{Prop:convergence} is now complete.
\end{proof}

We are finally ready to prove our main result:

\begin{proof}[Proof of Theorem~\ref{thm:main}]  As noted at the beginning of this section, Proposition~\ref{thm:pre-compactness} guarantees that every sequence $h_n\rt 0$ admits a subsequence so that $\pe^{h_n}$ converges in $C([-T,T];L^2_x(\R))$.  By Proposition~\ref{Prop:convergence}, the limiting function $\pe$ lies in $[C_tL^2_x\cap L^{\infty}_tH^1_x]([-T,T]\times\R)$ and solves the integral equation \eqref{duhamel}.

On the other hand, by the unconditional uniqueness result from Theorem~\ref{T:UCU}, this integral equation admits a unique $ L^{\infty}_tH^1_x$ solution.

This implies that all subsequential limits of $\phi^h$ agree, and so $\ph$ converges in $C([-T,T];L^2_x(\R))$ as $h\to0$ (without having to pass to a subsequence) and the resulting limit is the unique solution to \eqref{mKdV} with initial data $\pe_0$. Moreover, the uniform $H^1$ bounds \eqref{unif bdd} and \eqref{phi-H^1} allow us to upgrade this convergence to the modes of convergence described in \eqref{258} and \eqref{259}.
\end{proof}

\appendix
\section{Unconditional uniqueness for complex mKdV}\label{S:UCU}

The purpose of this appendix is to prove Theorem~\ref{T:UCU}.
For real-valued solutions, uniqueness under substantially weaker assumptions was shown in \cite{Kwon.Oh.Yoon2020}; however, this required a rather sophisticated and specialized argument.

As we will see, the arguments presented here are elementary and can be applied to a variety of gKdV-like equations.  For example, we could treat any linear combination of nonlinearities that are cubic or higher order.  Furthermore, the derivative may land on any term; it does not need to be outermost.  This is a key difference between the real and complex \eqref{mKdV} equations.

The construction of solutions in \cite{MR1211741} relies on contraction mapping in a class of functions satisfying some additional spacetime bounds.  Consequently, there can be only one solution satisfying these additional bounds.   In view of this, one may prove unconditional uniqueness by demonstrating that all $L^\infty_tH^1_x$ solutions to \eqref{mKdV} satisfy these additional spacetime bounds.  This can be achieved with minor modifications to what follows; however, the estimates we obtain below already suffice to demonstrate uniqueness via a simple Gronwall argument, so this is what we will present instead. 

To define the additional norms we use, we require a family of localizing functions.  For concreteness, we choose
\begin{align}\label{Po1}
\chi_j(x) := \frac{\sech(x-j)}{\bigl[\sum_k \sech^6(x-k)\bigr]^{\frac16}}  \quad\text{for each $j\in\Z$.}
\end{align}
Evidently, $\sum \chi_j^6(x) \equiv 1$ and the derivatives satisfy $|\partial_x^m\chi_j(x)| \lesssim_m \chi_j(x)$ for all $m\geq 0$.

\begin{proposition}\label{P:UCU}
Suppose $\phi\in L^\infty_tH^1_x$ is a solution to \eqref{mKdV} in the sense of \eqref{UCU Duhamel}.  Then for any $T>0$,
\begin{align}\label{UCU bounds}
\sum_{j\in\Z} \| \chi_j \phi \|_{L^\infty_{t,x}([-T,T]\times\R)}^4 < \infty \qtq{and}  \sup_j \| \chi_j \phi'' \|_{L^2_{t,x}([-T,T]\times\R)}^2 < \infty.
\end{align}
\end{proposition}

\begin{proof}
From the integral equation \eqref{UCU Duhamel}, one easily sees that $\phi\in C_t^{ }L^2_x$.  From this, it then follows that $\phi\in C_t^{ }H^s_x \cap C_t^1 H^{s-3}$ for any $s<1$.  In particular,
$$
\tfrac{d\ }{dt} \phi = -\phi''' \pm 6|\phi|^2\phi' \in C_t^{ }H^{s-3}_x \cap L^\infty_t H^{-2}_x .
$$

For each $0<\eps<\tfrac1{100}$, we define
\begin{align}\label{moll}
\phi_\eps(t,x):= [(1-\eps^2\partial^2)^{-2} \phi(t)](x) = \int \tfrac 1{4\eps} \bigl[ 1 + \tfrac{|x-y|}{\eps}\bigr]\ue^{-|x-y|/\eps} \phi(t,y)\,dy,
\end{align}
which satisfies
$$
\tfrac{d\ }{dt} \phi_\eps = -\phi_\eps''' \pm 6\big( |\phi|^2\phi'\bigr)_{\mkern -2mu\eps}\, .
$$
Here the nonlinearity is mollified in the same sense as \eqref{moll}.

By restricting $\eps<\tfrac1{100}$, we ensure that our mollifying kernel decays significantly faster than the localizing functions $\chi_j$.  Correspondingly, Schur's test shows that
\begin{align}\label{1354}
\| \chi_j^m f_\eps \|_{L^p(\R)} \lesssim \| \chi_j^m f \|_{L^p(\R)} 
\end{align}
uniformly for $1\leq p \leq \infty$, $f\in L^p(\R)$, and for any integer power $1\leq m\leq 3$.

The additional smoothness of $\phi_\eps$ allows us to integrate by parts with impunity and so verify that
\begin{align*}
\frac{d\ }{dt} \int \chi_j^2 |\phi_\eps|^2\,dx = \int -3\bigl(\chi_j^2\bigr)' |\phi_\eps'|^2 + \bigl(\chi_j^2\bigr)''' |\phi_\eps|^2
	\pm 12 \chi_j^2 \Re\Bigl\{ \overline{\phi}_\eps \big( |\phi|^2\phi'\bigr)_{\mkern -2mu\eps}\Bigr\} \,dx .
\end{align*}
Recalling \eqref{1354} and the fact that $\chi_j$ bounds its own derivatives, we deduce that
\begin{align*}
\sup_{|t|\leq T} \| \chi_j \phi_\eps(t) \|_{L^2_x}^2  \lesssim \| \chi_j \phi(0) \|_{L^2_x}^2
	+ \int_{-T}^T \bigl( 1 + \| \phi(s)\|_{H^1_x}^2 \bigr) \|\chi_j \phi(s) \|_{H^1_x}^2 \, ds .
\end{align*}
We now send $\eps\to0$ to obtain
\begin{align}\label{1369}
\sum_j \ \sup_{|t|\leq T} \| \chi_j \phi(t) \|_{L^2_x}^2  \lesssim_T \| \phi\|_{L^\infty_t H^1_x}^2 +  \| \phi\|_{L^\infty_t H^1_x}^4 .
\end{align}
This leads quickly to the first claim in \eqref{UCU bounds}.  Indeed, we simply combine it with the elementary Gagliardo--Nirenberg inequality
$$
\| \chi_j \phi(t) \|_{L^\infty_x}^4 \lesssim \| [\chi_j \phi]'(t) \|_{L^2_x}^2 \| \chi_j \phi(t) \|_{L^2_x}^2
	\lesssim \|\phi\|_{L^\infty_t H^1_x}^2 \| \chi_j \phi(t) \|_{L^2_x}^2 ,
$$
to see that
\begin{align}\label{infty bdd}
\sum_{j\in\Z} \| \chi_j \phi \|_{L^\infty_{t,x}([-T,T]\times\R)}^4 \lesssim \| \phi\|_{L^\infty_t H^1_x}^4 +  \| \phi\|_{L^\infty_t H^1_x}^6 .
\end{align}

We turn now to the second claim in \eqref{UCU bounds}, which has the form of a Kato smoothing estimate.  With this model in mind, we choose $w(x)$ to be a primitive (antiderivative) of $\chi_0^2(x)$ and then define $w_j(x):=w(x-j)$. Mirroring our earlier computation, we find that
\begin{align}\label{1379}
\frac{d\ }{dt} \int w_j |\phi_\eps'|^2\,dx = \int -3 |\chi_j \phi_\eps''|^2 + \bigl(\chi_j^2\bigr)'' |\phi_\eps'|^2
	\pm 12 w_j \Re\Bigl\{ {\overline{\phi}}_\eps^{\mkern+2mu\prime} \big( |\phi|^2\phi'\bigr)_{\mkern -2mu\eps}'\Bigr\} \,dx,
\end{align}
which we then integrate over the time interval $[-T,T]$.  Two of the resulting terms are estimated very easily:
 \begin{gather}\label{1381}
\sup_{j\in\Z}\ \sup_{t\in\R}\ \int w_j |\phi_\eps'(t)|^2\,dx \lesssim \| \phi \|_{L^\infty_t H^1_x}^2,  \\
\label{1382}
\sup_{j\in\Z}\ \int_{-T}^T \int \bigl(\chi_j^2\bigr)'' |\phi_\eps'(t)|^2 \,dx\,dt \lesssim T \| \phi \|_{L^\infty_t H^1_x}^2.
\end{gather}

To estimate the contribution of the last term from \eqref{1379}, we integrate by parts and use that $\chi_k^{6}$ form a partition of unity:
\begin{align*}
\biggl| \int_{-T}^T \int  w_j {\overline{\phi}}_\eps^{\mkern+2mu\prime} \big( |\phi|^2\phi'\bigr)_{\mkern -2mu\eps}'\Bigr\} \,dx\,dt \biggr|
&\lesssim T \| \phi \|_{L^\infty_t H^1_x}^4 + \sum_k \! \int_{-T}^T \| \chi_k \phi_\eps''\|_{L^2_x}
 	\bigl\| \chi_k^3 ( |\phi|^2\phi'\bigr)_{\mkern -2mu\eps}\bigr\|_{L^2_x} \,dt .
\end{align*}
We then employ \eqref{1354}, Cauchy--Schwarz, \eqref{infty bdd}, and the elementary inequality
$$
\sum_k \| \chi_k \phi'\|_{L^2_{t,x}([-T,T]\times\R)}^2 \lesssim T \| \phi\|_{L^\infty_t H^1_x}^2
$$
to deduce that
\begin{align*}
\sum_k \int_{-T}^T \bigl\| \chi_k \phi_\eps''\bigr\|_{L^2_x}  \bigl\| \chi_k^3 ( |\phi|^2 & \phi' \bigr)_{\mkern -2mu\eps} \bigr\|_{L^2_x} \,dt \\
&\lesssim \sum_k \bigl\| \chi_k \phi_\eps''\bigr\|_{L^2_{t,x}} \| \chi_k \phi \|_{L^\infty_{t,x}}^2 \bigl\| \chi_k \phi'\bigr\|_{L^2_{t,x}} \\
&\lesssim \sqrt{T} \, \Bigl(1 + \| \phi\|_{L^\infty_t H^1_x}^2 \Bigr) \| \phi\|_{L^\infty_t H^1_x}^3 \sup_k \, \bigl\| \chi_k \phi_\eps''\bigr\|_{L^2_{t,x}} \\
&\lesssim  T \delta^{-2} \, \Bigl(1 + \| \phi\|_{L^\infty_t H^1_x}^4 \Bigr) \| \phi\|_{L^\infty_t H^1_x}^6
	+ \delta^2 \sup_k \, \bigl\| \chi_k \phi_\eps''\bigr\|_{L^2_{t,x}}
\end{align*}
uniformly for $\delta\in(0,1)$.  Combining everything in this paragraph yields
\begin{align}\label{1383}
 \biggl| \int_{-T}^T \int  w_j {\overline{\phi}}_\eps^{\mkern+2mu\prime} \big( |\phi|^2\phi'\bigr)_{\mkern -2mu\eps}' \,dx\,dt \biggr|
 &\lesssim T \delta^{-2} \Bigl(1 + \| \phi\|_{L^\infty_t H^1_x}^{10} \Bigr) + \delta^2 \sup_k \| \chi_k \phi_\eps''\|_{L^2_{t,x}}
\end{align}
uniformly for $\delta\in(0,1)$.

Combining \eqref{1379}, \eqref{1381}, \eqref{1382}, and \eqref{1383}, we deduce that
\begin{align}\label{1414}
\sup_j \, \bigl\| \chi_j \phi_\eps''\bigr\|_{L^2_{t,x}}^2 \lesssim (1 + T) \delta^{-2} \Bigl(1 + \| \phi\|_{L^\infty_t H^1_x}^{10} \Bigr) + \delta^2 \sup_k \| \chi_k \phi_\eps''\|_{L^2_{t,x}}
\end{align}
uniformly for $\delta\in(0,1)$.  Choosing $\delta$ small enough to overcome the (absolute) implicit constant and then sending $\eps\to 0$, yields the second bound in \eqref{UCU bounds}.
\end{proof}

\begin{proof}[Proof of Theorem~\ref{T:UCU}]
As noted earlier, the existence of solutions $\phi\in L^\infty_t H^1_x$ to \eqref{UCU Duhamel} follows from the arguments in \cite{MR1211741}.  We need only address uniqueness.

Using Proposition~\ref{P:UCU}, we see that for any such solution $\phi$ and any $T>0$,
\begin{align*}
\bigl\| |\phi|^2 \phi'' \|_{L^2_{t,x}([-T,T]\times\R)}^2 &= \sum_j \int_{-T}^T \int \bigl|\chi_j\phi''\bigr|^2 \bigl|\chi_j \phi\bigr|^4 \,dx\,dt \\
&\lesssim \sup_j \|\chi_j\phi''\|_{L^2_{t,x}}^2  \cdot \sum_j  \|\chi_j \phi\|_{L^\infty_{t,x}}^4 < \infty
\end{align*}
and consequently, $|\phi|^2 \phi'' \in L^1_t L^2_x([-T,T]\times\R)$.  Likewise, any $L^\infty_tH^1_x$ solution $\phi$ satisfies
$$
\bigl\| (|\phi|^2)' \phi' \|_{L^{6/5}_t L^1_x([-T,T]\times\R)} \lesssim T^{6/5} \| \phi \|_{L^\infty_t H^1_x}^3 < \infty. 
$$
Taken together these estimates control the derivative of the nonlinearity in dual Strichartz spaces.  Applying the Strichartz inequality  \cite{GV,Kenig.Ponce.Vega1991} to \eqref{UCU Duhamel}, we deduce that
\begin{align}\label{1439}
\bigl\| \phi' \|_{L^6_tL^\infty_x} &\lesssim \| \phi_0\|_{H^1} +  \bigl\| |\phi|^2 \phi'' \|_{L^1_tL^2_x}
	+ \bigl\| (|\phi|^2)' \phi' \|_{L^{6/5}_t L^1_x} < \infty,
\end{align}
where all spacetime norms are over $[-T,T]\times\R$.

Suppose now that we have two $L^\infty_t H^1_x$ solutions $\phi$ and $\psi$ to \eqref{UCU Duhamel} with the same initial data $\phi_0$.  By the analysis above, we know that $\phi',\psi' \in L^6_tL^\infty_x([-T,T]\times\R)$ for every $T>0$.

Mimicking the mollification argument in Propostion~\ref{P:UCU}, we find that
\begin{align*}
\tfrac{d\ }{dt} \bigl\| \phi-\psi \|_{L^2_x}^2 &= \int \mp 6 \bigl(|\phi|^2)' |\phi-\psi|^2 
	\pm 12\Re \bigl\{ [\overline{\phi}-\overline\psi][|\phi|^2-|\psi|^2]\psi'\bigr\}\,dx
\end{align*}
and so that
\begin{gather*}
\Bigl| \tfrac{d\ }{dt} \bigl\| [\phi-\psi](t) \|_{L^2_x}^2 \Bigr| \leq C(t) \bigl\|[\phi-\psi](t)\bigr\|_{L^2_x}^2 
\end{gather*}
where the growth rate $C(t)$ satisfies
\begin{gather*}
C(t)\lesssim \Bigl[\|\phi(t)\|_{L^\infty_x} + \|\psi(t)\|_{L^\infty_x}\Bigl]\Bigr[\|\phi'(t)\|_{L^\infty_x}+\|\psi'(t)\|_{L^\infty_x} \Bigr].
\end{gather*}
In view of \eqref{1439} and $\phi\in L^\infty_t H^1_x$, we see that $C(t)$ is integrable over any finite time interval; thus Gronwall's inequality shows that $\phi\equiv\psi$.
\end{proof}


\end{document}